%% file: main.tex
\documentclass{article}
\usepackage[utf8]{inputenc}
\usepackage{hyperref}
\usepackage{amsmath, amsthm, amssymb}
\usepackage{amsfonts, mathtools}
\usepackage{bm, bbm}
\usepackage{enumitem}

\title{Coboundaries and eigenvalues of morphic subshifts}
\author{Paul MERCAT\\Aix-Marseille University, CNRS, I2M, Marseille, France}

\input{commands.tex}

\begin{document}

\maketitle

\begin{abstract}
    We define a morphic subshift as a subshift generated by the image of a substitution subshift by another substitution.
    In other words, it is the subshift associated with a ultimately periodic directive sequence.
    We present an efficient algorithm for computing eigenvalues of morphic subshifts using coboundaries.
    We show that continuous eigenvalues of $\S$-adic subshifts for primitive directive sequences are always associated with some coboundary.
    For morphic subshifts, we provide a characterization of the dimension of the $\Q$-vector space generated by eigenvalues.
    Additionally, if the substitutions are unimodular and without coboundaries,
    we give a necessary and sufficient condition for weak mixing of the subshift.
\end{abstract}

\tableofcontents

\section{Introduction}

Eigenvalues of topological or measure-preserving dynamical systems are important invariants.
For example a rational (additive) eigenvalue indicates that the system is partitioned with a cyclic dynamic on pieces,
while an irrational eigenvalue implies a rotation of the circle as a factor.
Additionally, a minimal system is weakly mixing if and only if the set of its additive eigenvalues is $\Z$.

Coboundaries can naturally occur.
Indeed, it arose when we where studying a problem of Novikov with Pascal Hubert and Olga Paris-Romaskevich.
We found a renormalization of our dynamical system that allows us to conjugate it with a $\S$-adic subshift
with a coboundary (see Example~\ref{ex:CET4}).

In 1986, a seminal paper by Host (see~\cite{Host86}) characterized eigenvalues of a substitution subshift
with coboundaries and showed the surprising result that every eigenvalue is associated with a continuous eigenfunction.

Ten years later, Ferenczi-Mauduit-Nogueira (see~\cite{FMN96}) provided a completely different algebraic characterization of eigenvalues of substitution subshifts,
allowing for explicit computation. 
However their characterization is complicated and they even made a calculation error
in one of the two examples they provided (see Example~9.1 in~\cite{Me23} for more details).
More recently, I proposed a simpler algorithm to compute these eigenvalues (see~\cite{Me23}),
based on Host's characterization and using a algorithm for proprification due to Durand and Petite (see~\cite{DP20}).
However, this algorithm only works for pseudo-unimodular substitutions, and the proprification algorithm
can sometimes produce very large matrices, that can complicate the computation of eigenvalues.

Much work has been done to extend Host's result, for example, for linearly recurrent subshifts (see~\cite{CDHM2003}),
or for minimal Cantor systems (see~\cite{BDM05} and~\cite{DFM2019}), or particular examples (see~\cite{CFM08}).
But they always assume that substitutions are left-proper, which avoid coboundaries.
In~\cite{BBY22}, Berthé-Bernales-Yassawi propose a generalization of Host's result using coboundaries,
but they make very strong assumptions 
and their definition of coboundary is complex and not very general.

I start in Section~\ref{sec:defs} by giving notations and definitions used in the article.
Then, in Section~\ref{sec:gen:crit}, I give a general criterion ensuring the existence of eigenvalues for a subshift.
In Section~\ref{sec:Sadic}, I partially extend Host's result and show that any continuous eigenvalue of a primitive directive sequence is associated with some coboundary.
I use it in Section~\ref{sec:morphic} to propose a new algorithm to compute eigenvalues of morphic subshifts that works without proprification and for more general subshifts.
I give an explicit description of the set of eigenvalues, and algorithms to compute it and to check aperiodicity and recognizability.
In Section~\ref{sec:wm}, I give a characterization of the dimension of the $\Q$-vector space spanned by eigenvalues of a morphic subshift.
If moreover the iterated substitution is unimodular without coboundary,
I provide a necessary and sufficient condition for the subshift to be weakly mixing, by looking at the action of some Galois group on pairs of algebraic numbers.
I finish by providing examples and a way to construct examples with non-trivial coboundaries in Section~\ref{sec:exs}, with some examples coming from geometrical problems.

\section{Definitions and notations} \label{sec:defs}

\subsection{Words and worms}

An \defi{alphabet} is defined as a finite set.
The set of finite words over an alphabet $\A$ is denoted $\A^*$.
It forms a monoid under concatenation.
We use $\abs{w}$ to represent the length of a word $w$, and $\abs{w}_a$ to indicate the number of occurrences of the letter $a$ in $w$.
We denote $w_{[i, j)}$ the subword of $w$ between indices $i$ (inclusive) and $j$ (exclusive).
The \defi{abelian vector} of a word $w \in \A^*$ is defined as $\ab(w) = (\abs{w}_a)_{a \in \A} \in \Z^A$.
We denote $e_a = \ab(a)$, so that $(e_a)_{a \in \A}$ forms the canonical basis of $\R^\A$.

The \defi{worm} of an infinite word $x \in \A^\N$ is
$W(x) = \{ \ab(x_{[0,n)}) \mid n \in \N \} = \{\ab(p) \mid p \text{ prefix of } x \}$.

\subsection{Subshift}

Given an alphabet $\A$, we equip the sets of infinite words $\A^\N$ and bi-infinite words $\A^\Z$ with the product toplogy.
A \defi{subshift} is a compact subset $X$ of $\A^\N$ or $\A^\Z$ endowed with the \defi{shift map} $S : \A^N \to \A^\N$ or $S : \A^\Z \to \A^\Z$. We refer to the subshift as \defi{infinite} or \defi{bi-infinite}.
A subshift $(X,S)$ is \defi{minimal} if every orbit is dense in $X$.
The \defi{language} of a subshift $(X,S)$ is the set $\L_X$ of finite factors of elements of $X$.

An (additive) \defi{eigenvalue} of a subshift $(X,S)$ is a real number $t \in \R$ such that there exists a function $f: X \to \R/\Z \in L^2(\mu)$, where $\mu$ is an $S$-invariant measure,
such that $f \circ S = f + t$. If $f$ is continuous, we refer to $t$ as a \emph{continuous eigenvalue}.

Let $(X,S)$ be a bi-infinite subshift. Then, we have an associated infinite subshift $(\tilde X, S)$
where elements of $\tilde X$ are the right-infinite part of elements of $X$.
We have the following (see Proposition~2.1 in~\cite{BBY22}).

\begin{prop} \label{prop:inf:bi-inf}
    Let $(X,S)$ be a minimal subshift.
    Then, $(X,S)$ and $(\tilde X,S)$ have the same continuous eigenvalues.
\end{prop}

A subshift $(X,S)$ is \defi{linearly recurrent} if there exists a constant $K$ such that
for every $x \in X$ and for every $k \in \N$, there exists $0 < n < Kk$ such that $x_{[0,k)} = x_{[n,n+k)}$.

\subsection{Return words and coboundaries} \label{ss:cob}

A \defi{return word} of a subshift $X$ is a word $w$ such that $w w_0 \in \L_X$ and $\abs{w}_{w_0} = 1$, where $w_0$ denotes the first letter of $w$.
We refer to the subspace $R(X)$ (or just $R$) of $\R^\A$, generated by $\ab(w)$ for return words $w$, as the \defi{return space}.
The \defi{return space on letter $a$} is denoted $R_a(X)$ and is the subspace of $R(X)$ generated by return words on letter $a$.
A \defi{coboundary} of a subshift $X$ over an alphabet $\A$ is a morphism $c : \A^* \to \R/\Z$
such that $c(w) = 0$ for every return word $w$.
Note that the data provided by such a morphism is equivalent to that of a map $\tilde c : \R^\A \to \R/\Z$
such that $\tilde c \circ \ab = c$. Throughout the following, we use the same notation to represent both maps.
A \defi{real coboundary} of a subshift $X$ over an alphabet $\A$ is a linear form $c : \R^\A \to \R$ such that
$c(\ab w) = 0$ for every return word $w$.
As for coboundaries, we can see a real coboundary as a map $\A^* \to \R$.

\begin{lem} \label{lem:rep}
    Let $(X,S)$ be a minimal subshift.
    Any coboundary $c : \A^* \to \R/\Z$ has a representative $\tilde c : \R^\A \to \R$ which is a real coboundary.
\end{lem}

\begin{proof}
    Given a coboundary $c$, we construct a directed graph with edges labeled by $\A$ and vertices in $\R/\Z$,
    as follows:
    Start from vertex $0$ and choose some letter $a \in \L_X$.
    Add an edge $0 \xrightarrow{a} c(a)$.
    Choose a letter $b \in \A$ such that $ab \in \L_X$, then add the edge $c(a) \xrightarrow{b} c(ab)$.
    Choose a letter $c \in \A$ such that $abc \in \L_X$, then add the edge $c(ab) \xrightarrow{c} c(abc)$. 
    Continue this process until every letter has been considered.
    
    Note that each letter appears in the graph only once, and every word of $\L_X$ starting by letter $a$ corresponds to a path from $0$ in the automaton.
    Next, for each vertex of the graph, choose a representative in $\R$.
    We define $\tilde c (a) := e-s$ for each edge $s \xrightarrow{a} e$. 
    Then $\tilde c$ defines a real coboundary and $c = \tilde{c}$ mod $\Z$.
\end{proof}

We call \defi{coboundary space} of a subshift $(X,S)$ the vector space of real coboundaries.
In other words, it is $R(X)^\circ$,
where $E^\circ = \{ w \in E^* \mid \forall x \in E,\ wx = 0 \}$ denotes the annihilator of $E$.
Note that it is completely determined by a graph that we call \defi{coboundaries graph} with $d+1$ vertices where $d$ is the dimension of the coboundary space.

\begin{ex}
    Consider the periodic shift generated by the periodic word $(ab)^\omega$.
    Then, $(1/2, 1/2)$ is a coboundary.
    The corresponding graph is
    
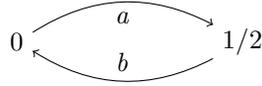
\begin{center}
        \begin{tikzpicture}
            \node at (0,0) (0) {$0$};
            \node at (3cm,0) (1) {$1/2$};
            
            \draw[->] (0) to [bend left] node[below]{a} (1);
            \draw[->] (1) to [bend left] node[above]{b} (0);
        \end{tikzpicture}
    \end{center}
    Next, we obtain the linear form $(1/2, -1/2) : \R^2 \to \R$ that annihilates return words $\{ab, ba\}$.
    The coboundary space is the one-dimensional subspace generated by this linear form, and the coboundaries graph is this graph.
\end{ex}

Coboundaries naturally extend to infinite words:

\begin{lem}
    Let $(X,S)$ be a minimal subshift, and let $c : \A^* \to \R/\Z$ be a coboundary.
    Then we can extend $c : X \cup \A^* \to \R/\Z$ such that $c(x_0) = c(x) - c(Sx)$
    for every $x \in X \cup \A^+$.
\end{lem}

\begin{proof}
    As in the proof of previous lemma, consider a coboundary graph associated with the coboundary $c$.
    We define $c$ on $X$ by $c(x) = -v$ where $v \xrightarrow{x_0} \cdot$ is the unique arrow
    labeled by letter $x_0$ in the graph. Then we have $c(x_0) = c(x) - c(S x)$.
\end{proof}

Note that it also works with real coboundaries.
The extension of $c$ to $X$ is unique if we set the value $c(x)$ at a point $x \in X$.

\subsection{Matrices} \label{ss:matrix}

A matrix $M \in M_{d}(\N)$ is
\begin{itemize}
    \item \defi{primitive} if there exists $n \geq 1$ such that $M^n$ has strictly positive coefficients,
    \item \defi{irreducible} if its characteristic polynomial is irreducible,
    \item \defi{unimodular} if its determinant is $\pm 1$,
    \item \defi{pseudo-unimodular} if the product of its non-zero eigenvalues is $\pm 1$, counting multiplicities.
    \item \defi{Pisot} if it has one eigenvalue $>1$ and every other eigenvalue has a modulus strictly less than $1$.
\end{itemize}

We denote $M > 0$ if every coefficient of $M$ is strictly positive.
We define the \defi{pseudo-inverse} of a square matrix $M$ as the matrix $M^{-1}$ obtained by inversion
of every invertible Jordan block and taking the transpose of the other ones in the Jordan decomposition of the matrix.
When the matrix $M$ is invertible, this is the usual inverse of the matrix.
In restriction to the eventual image $I = \bigcap_{n \in \N} M^n \R^A$, the matrix $\restriction{M}{I}$ is invertible
and its inverse is the pseudo-inverse $\restriction{M^{-1}}{I}$.

A number $\beta \in \R$ is \defi{Perron} if it is an algebraic integer $\beta > 1$ such that every Galois conjugate has modulus strictly less than $\beta$. It is \defi{Pisot} if moreover Galois conjugates have modulus strictly less than $1$.

We have the following well-known theorem.
\begin{thm}(Perron-Frobenius)
    If $M \in M_d(\N)$ is a primitive matrix, then its greatest eigenvalue is simple and is a Perron number.
    Moreover, it has an associated eigenvector with strictly positive coordinates.
\end{thm}
We call this greatest eigenvalue the \defi{Perron eigenvalue} and an associated eigenvector is a \defi{Perron eigenvector}.

For sequences of invertible matrices, we have the following.

\begin{thm}[Osseledet's Theorem]
    Let $\mu$ be an ergodic invariant measure on the full shift $(\M^\N, S)$, where $\M \subset GL_d(\R)$ is finite.
    For almost every sequence $(M_n)_{n \in \N} \in \M^\N$,
    there exists $0 < r \leq d$, $\theta_1 > ... > \theta_r$ and subspaces $\R^d = E_1 \supsetneqq ... \supsetneqq E_{r-1} \supsetneqq E_r = \{0\}$ such that
    \[
        v \in E_i \backslash E_{i+1} \Longleftrightarrow \lim_{n \to \infty} \frac{1}{n} \log \norm{v M_{[0,n)}} = \theta_i.
    \]
\end{thm}

We call $\theta_i$ a \defi{Lyapunov exponent} of $(M_n)$ and $E_i$ the corresponding \defi{Lyapunov space}.

\vide
{
    \begin{define}
        We say that a sequence of matrices $(M_n)$ has the \defi{fast convergence property} if we have
        \[
            v M_{[0,n)} \xrightarrow{n \to \infty} 0 \Longrightarrow \sum_{n \in \N} \norm{v M_{[0,n)}} < \infty.
        \]
    \end{define}
    
    Note that if $(M_n)$ is hyperbolique (i.e. $0$ is not a Lyapunov exponent) then it has the fast convergence property.
    If $(M_n)$ has a Lyapunov exponent $\theta_i = 0$ such that there exists a subspace $F$ with $E_i = F + E_{i+1}$
    and $\exists C > 0,\ \forall n \in \N,\ M_n(F) \subset F$ and $C < \norm{\restriction{M_{[0,n)}}{F}} < 1/C$, then $(M_n)$ also has the fast convergence property.
    By Theorem~\ref{thm:allcont} every eigenvalue of $(X_{\bm \sigma}, S)$ is continuous.
}

\subsection{Substitution}

Let $\A$ and $\B$ be two alphabets.
A \defi{word morphism} is a morphism $\sigma: \A^* \to \B^*$ for the concatenation of words.
It is \defi{non erasing} if $\forall a \in \A,\ \abs{\sigma(a)} \geq 1$.
A \defi{substitution} is a non-erasing word morphism.
We call $\A$ the \defi{domain alphabet} and $\B$ the \defi{codomain alphabet} of $\sigma: \A^* \to \B^*$.
If both alphabets of a substitution $\sigma$ are the same, we denote it $\A_\sigma$.

The \defi{matrix} of a substitution $\sigma : \A^* \to \B^*$ is the unique matrix $M_\sigma$ (or $\ab(\sigma)$) satisfying $\ab(\sigma(w)) = M_\sigma \ab(w)$ for every word $w \in \A^*$.

For example, the matrix of the Tribonnacci substitution $a \mapsto ab,\ b \mapsto ac,\ c \mapsto a$ is
$\begin{pmatrix} 1 & 1 & 1 \\ 1 & 0 & 0 \\ 0 & 1 & 0 \end{pmatrix} $.

We say that a substitution is primitive, irreducible, Pisot, pseudo-unimodular, etc...
if its matrix has the corresponding property and if domain and codomain alphabets are the same.

The \defi{language of a substitution} $\sigma$ is the set $\L_\sigma$ of factors of $\sigma^n(a)$, for $a \in \A$ and $n \in \N$, where $\A$ is the domain alphabet of $\sigma$.
The infinite (resp. bi-infinite) \defi{subshift of a substitution} $\sigma$ is the set of infinite (resp. bi-infinite) words
whose factors are in $\L_\sigma$.
Note that the language of a primitive substitution is the language of its subshift.

A \defi{substitution subshift} is the subshift of a primitive substitution.
A \defi{morphic subshift} is the subshift generated by the image by a substitution of a substitution subshift.


Let $X$ be a subshift over an alphabet $\A$,
and let $\sigma: \A^* \to \B^*$ be a morphism.
Let $Y$ be the shift generated by $\sigma(X)$.
Every element of $Y$ can be written in the form $S^k \sigma(x)$, with $x \in X$ and $0 \leq k < \abs{\sigma(x_0)}$.
We say that $\sigma$ is \defi{recognizable} in $X$ if such a representation is unique for every element of $Y$.
We have the following (see~\cite{Mosse1996} Theorem~1).

\begin{thm}[Mossé] \label{thm:mosse}
    Let $\sigma$ be a primitive aperiodic substitution.
    Then $\sigma$ is recognizable in the bi-infinite subshift $X_\sigma$.
\end{thm}

Note that a primitive substitution is aperiodic if its Perron eigenvalue is not an integer by Remark~\ref{rem:per}.


The \defi{prefix-suffix automaton} of a substitution $\sigma: \A^* \to \A^*$ is a labeled directed graph
with vertices the alphabet $\A$ and with edges $a \xrightarrow{p,s} b$
if and only if $\sigma(a) = pbs$.
The \defi{prefix automaton} is the same automaton where labels $p,s$ are replaced by $p$.
We denote $a_{n+1} \xrightarrow{p_n} a_n ... \xrightarrow{p_0} a_0$ if $a_{k+1} \xrightarrow{p_k} a_k$ for every $0 \leq k \leq n$.

Every element $x$ of the bi-infinite subshift $X_\sigma$ has a \defi{prefix-suffix decomposition}
\[
    ...\sigma^3(p_3)\sigma^2(p_2)\sigma(p_1)p_0 \cdot a_0 s_0 \sigma(s_1) \sigma^2(s_2) \sigma^3(s_3)...
\]
This decomposition is not always equal to $x$ since it may not define a bi-infinite word.
If $\sigma$ is recognizable, such a decomposition is unique.
Thus, it permits to define maps $p_n : X_\sigma \to \A^*$.

The following lemma gives a property of real coboundaries of a substitution subshift.

\begin{lem} \label{lem:cob:unity}
    Let $\sigma$ be a primitive substitution.
    The coboundary space $R(X_\sigma)^\circ$ is included in the sum of generalized left-eigenspaces of $M_\sigma$
    for eigenvalues roots of unity and zero.
\end{lem}

In particular, if $M_\sigma$ has no eigenvalue root of unity nor zero, then $X_\sigma$ has no non-trivial coboundary.

\begin{proof}
    We have $R^\circ M_\sigma \subseteq R^\circ$,
    since the image by $\sigma$ of a return word is a product of return words.
    Let $c \in R^\circ$. Then $(c M_\sigma^n)_{n \in \N}$ is bounded.
    Indeed, for every letter $a \in \A$, $\sigma^n(a) \in \L_\sigma$, and $c(\L_\sigma)$ is bounded.
    Thus, $c$ is orthogonal to every generalized eigenspace of $M_\sigma$ for an eigenvalue of modulus $>1$.
    Moreover, $R$ is rational, thus $R^\circ$ is stable by field morphisms,
    so it is orthogonal to every generalized eigenspace of $M_\sigma$ for an eigenvalue not root of unity nor zero.
\end{proof}

\subsection{Directive sequences}

A \defi{directive sequence} is an infinite sequence $\bm \sigma = (\sigma_n)_{n \in \N}$ of substitutions
such that $\forall n \in \N$, $\sigma_n : \A_{n+1}^* \to \A_n^*$ (i.e. domains and codomains are compatible to compose $\sigma_n \sigma_{n+1}$).
We use the standard notation $\sigma_{[i,j)} = \sigma_i ... \sigma_{j-1}$.

A directive sequence is \defi{finitary} if its set of substitutions is finite.
We say that a directive sequence $\bm \sigma$ is \defi{primitive} if
$\forall k \in \N$, $\exists n \in \N$, $M_{[k,k+n)} > 0$.
It is \defi{strongly primitive} if $\exists n \in \N,\ \forall k \in \N,\ M_{[k,k+n)} > 0$.

The \defi{language of a directive sequence} $\bm \sigma$ is the set $\L_{\bm \sigma}$ of factors of $\sigma_{[0,n)}(a)$, for $n \in \N$, $a \in \A_n$.
The infinite (resp. bi-infinite) \defi{subshift of a directive sequence} $\bm \sigma$ is the set $X_\sigma$ of infinite (resp. bi-infinite) words whose factors are in the language of $\bm \sigma$.

We denote $S : \S^N \to \S^\N$ the shift of directive sequences over substitutions $\S$.
We denote $X_{\bm \sigma}^{(n)}$ (or just $X^{(n)}$) the subshift $X_{S^n \bm \sigma}$.
Idem for $\L_{\bm \sigma}^{(n)}$.

A directive sequence $\bm \sigma$ is \defi{recognizable} if for all $n \in \N$, $\sigma_n$ is recognizable in $X^{(n+1)}$.
Note that if substitutions of $\bm \sigma$ are unimodular and $X_{\bm \sigma}$ is aperiodic,
then $\bm \sigma$ is recognizable by Theorem~3.1 in~\cite{BSTY19}.

The \defi{prefix-suffix diagram} is a natural generalization of prefix-suffix automaton to directive sequences,
but it is no longer an automaton; instead it is a Bratelli diagram with vertices $\A_n$ at level $n$
(alphabets are considered to be disjoint)
and edges $a_{n+1} \xrightarrow{p_n,s_n} a_{n}$ if and only if $\sigma_n(a_{n+1}) = p_n a_n s_n$.
The \defi{prefix diagram} is identical but with only $p_n$ on labels of edges.
Every element of $X_{\bm \sigma}$ admits a \defi{prefix-suffix decomposition}
\[
    ..\sigma_{[0,3)}(p_3)\sigma_{[0,2)}(p_2)\sigma_0(p_1)p_0 \cdot a_0 s_0 \sigma_0(s_1) \sigma_{[0,2)}(s_2) \sigma_{[0,3)}(s_3)...
\]
where $...\xrightarrow{p_n,s_n} a_n ... \xrightarrow{p_0,s_0} a_0$ is an infinite path
in the prefix-suffix diagram.
But this decomposition is not always equal to the original bi-infinite word: if $(p_n)$ or $(s_n)$ are ultimately constant to $\epsilon$, then it does not gives a bi-infinite word. See~\cite{CS} for more details.
When the directive sequence $\bm \sigma$ is recognizable, such decomposition is unique so maps $p_n : X^{(n)} \to \A_n^*$
and $a_n: X^{(n)} \to \A_n$ are well-defined.

\section{A general sufficient condition to be an eigenvalue} \label{sec:gen:crit}

We give a general sufficient condition for a number to be an eigenvalue of a subshift.

\begin{prop} \label{prop:gs}
    Let $(X,S)$ be a minimal subshift over a finite alphabet $\A$.
    Let $\varphi: \A^* \to \R$ be a morphism such that $\varphi(\L_X)$ is bounded
    and such that $\forall a, b \in \A$, $\varphi(a) = \varphi(b)$ mod $1$.
    Then, $\varphi(a)$ is a continuous eigenvalue of $(X,S)$.
\end{prop}

Note that this proposition can be reformulated with $\varphi$ replaced by a linear form of $\R^\A$, with the image of $\ab(\L_X)$ bounded.
And the hypothesis on $\varphi$ can be replaced by $\varphi(W(x))$ is bounded for one $x \in X$.

To prove the proposition, we need the following.

\begin{lem} \label{lem:ext:cont}
    Let $X$ be a minimal shift over a finite alphabet $\A$.
    Let $\varphi \in (\R^\A)^*$ be a linear form such that $\varphi(\L_X)$ is bounded.
    Let $x \in X$.
    Then, the map defined by $\forall n \in \N$, $f(S^n x) = \varphi(\ab x_{[0,n)})$
    can be extended by continuity to $X$.
\end{lem}

\begin{proof}
    If $x = w^\omega$ is periodic, then the hypothesis give $\varphi(\ab w) = 0$, so the map $f$ is well-defined.
    If the shift $X$ is periodic, then it is equal to an orbit and the result is obvious.
    Otherwise, we follow the proof of Lemma~{8.2.5} in~\cite{AkiyamaMercat2020}.
\end{proof}

\begin{proof}[Proof of Proposition~\ref{prop:gs}]
    Let $t = \varphi(a) \in \R/\Z$.
    Let $x \in X$.
    We define a map $f$ on the orbit of $x$ by $f(S^n x) = n t = \varphi(x_{[0,n)})$ mod $1$.
    By Lemma~\ref{lem:ext:cont}, $f$ can be extended by continuity to $X$.
    Then, $f$ is a continuous eigenfunction. 
\end{proof}

\begin{cor}
    Let $(X,S)$ be a minimal subshift. Let $x \in X$.
    If $W(x)$ is at bounded distance of a line, then coordinates of the direction vector of sum $1$ are
    continuous eigenvalues of $(X,S)$.
\end{cor}

In other words if the discrepancy is bounded (see~\cite{Adam04} for a definition), then frequencies of letters are
continuous eigenvalues of the subshift.

\begin{ex}
    Let $u = (ab)^\omega$ or let $u$ be the Thue-Morse sequence.
    $W(u)$ is at bounded distance of the line with direction vector $(1,1)$, thus $1/2$ is an eigenvalue of the generated subshift.
\end{ex}

\begin{cor}[Pisot morphic subshift]
    Let $\sigma$ be a Pisot substitution and $\tau$ be any substitution.
    Let $v$ be a Perron eigenvector of $\ab(\sigma)$, then coordinates of $\frac{\ab(\tau) v}{\norm{\ab(\tau) v}_1}$
    are continuous eigenvalues of the subshift $(X_{\tau \sigma^\omega},S)$.
\end{cor}

In other words, frequencies of letters are eigenvalues for such subshifts.
This result is due to Host (see (6.2) in~\cite{Host86}) in the particular case where $\tau$ is the identity.

\begin{proof}
    Let $f_a$ be the coordinate of $\frac{\ab(\tau) v}{\norm{\ab(\tau) v}_1}$ for letter $a$, which is in fact the frequency of $a$.
    Let $\phi = f_a (1, ..., 1) - e_a$.
    Let us show that $\phi(\ab \L_{\tau \sigma^\omega})$ is bounded.
    It is equivalent to show that $\phi (W(u)) $ is bounded for some $u \in X_{\tau \sigma^\omega}$.
    We take for example the image by $\tau$ of a periodic point of $\sigma$.
    Elements of $W(u)$ are of the form
    $\ab(p_0) + \sum_{n=0}^{N-1} M_\tau M_\sigma^n \ab(p_{n+1})$
    for a path $a_{N+1} \xrightarrow{p_N} ... \xrightarrow{p_0} a_0$ in the prefix diagram of $\tau \sigma^\omega$.
    Thanks to the hypothesis, we have $\phi M_\tau M_\sigma^n \xrightarrow[n \to \infty]{} 0$,
    since $\phi M_\tau M_\sigma^n v = 0$,
    and this convergence is exponentially fast.
    Thus,
    $\phi(W(u))$
    is uniformly bounded.
\end{proof}

\begin{cor}[$\S$-adic with coboundary] \label{cor:Sadic}
    Let $\bm \sigma$ be a finitary primitive directive sequence.
    Let $\pi : \R^{\A_0} \to \R$ and $c : \R^{\A_0} \to \R$ be such that 
    \[
        \sum_{n \in \N} \norm{ (\pi + c) M_{[0,n)} } < \infty,
    \]
    $c(\L_{\bm \sigma})$ is bounded and for every $a,b \in \A_0$, $\pi(a) = \pi(b)$ mod $\Z$.
    Then $\pi(a)$ is a continuous eigenvalue of $(X_{\bm \sigma}, S)$.
\end{cor}

\begin{proof}
    The primitivity gives that $(X_{\bm \sigma}, S)$ is minimal.
    The hypothesis tells us that $(\pi+c)(\L_{\bm \sigma})$ is bounded, by looking at prefix-suffix decomposition of elements of $\L_{\bm \sigma}$, so $\pi(\L_{\bm \sigma})$ is also bounded.
    Thus, $\pi$ satisfies the hypothesis of Proposition~\ref{prop:gs} and it gives the conclusion.
\end{proof}

\begin{ex}[$\S$-adic with coboundary] 
    For any Bernouilli measure, almost every primitive directive sequence $\bm \sigma$
    over the following set of substitution has a set of eigenvalues
    which is the $\Z$-module generated by frequencies of letters.
    \[
        \S = \left\{\begin{array}{l}
        a \mapsto cb\\
        b \mapsto c\\
        c \mapsto ab
        \end{array}, \begin{array}{l}
        a \mapsto bc\\
        b \mapsto bcb\\
        c \mapsto ca
        \end{array}\right\}
    \]
    We easily check that $\bm \sigma$ is strongly primitive since every product of length $4$ of incidence matrices is positive.
    Moreover, both substitutions stabilize the graph
    
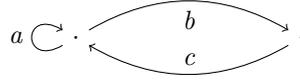
\begin{center}
        \begin{tikzpicture}
            \node at (0,0) (0) {$\cdot$};
            \node at (3cm,0) (1) {$\cdot$};
            
            \draw[->] (0) to [in=150, out=210, looseness=8] node[left]{a} (0);
            \draw[->] (0) to [bend left] node[below]{b} (1);
            \draw[->] (1) to [bend left] node[above]{c} (0);
        \end{tikzpicture}
    \end{center}
    thus $\begin{pmatrix} 0 & 1 & -1 \end{pmatrix}$ is a coboundary of $(X_{\bm \sigma}, S)$
    and it is a common left-eigenvector of incidence matrices for eigenvalues $1$ or $-1$.
    Furthermore, both substitutions are unimodular, so the sum of Lyapunov exponents is zero: $\theta_1+\theta_2+\theta_3 = 0$.
    We deduce that $\theta_1 > \theta_2 = 0 > \theta_3$.
    Moreover, for every integer vector $w \in \Z^\A$, there exists $x,y \in \R$ such that
    $x\begin{pmatrix} 1 & 1 & 1 \end{pmatrix} + y\begin{pmatrix} 0 & 1 & -1 \end{pmatrix} - w \in E_3$.
    The hypothesis of Corollary~\ref{cor:Sadic} is satisfied with the coboundary $c = -y\begin{pmatrix} 0 & 1 & -1 \end{pmatrix}$ and $\pi = x\begin{pmatrix} 1 & 1 & 1 \end{pmatrix} - w$.
    Thus, coordinates of $v$ are continuous eigenvalues, 
    where $v$ is the vector of sum $1$ such that
    $\bigcap_{n \in \N} M_{[0,n)} \R_+^\A = \R_+ v$.
    Moreover, by Theorem~\ref{thm:nec} thereafter there are no more continuous eigenvalues.
    In other words, the set of continuous eigenvalues is the $\Z$-module generated by frequencies of letters.
    In particular, $(X_{\bm \sigma}, S)$ is not weakly mixing.
\end{ex}


Proposition~\ref{prop:gs} can be extended to find eigenvalues of $\R$-actions:

\begin{prop}
    Let an $\R$-action on a space $Y$ with a Borel invariant probability measure.
    Let $Z \subset Y$ be a subset such that $\forall y \in Y$, $\exists t \in \R_+^*$, $t.y \in Z$,
    and with return times $(h_a)_{a \in \A} \in (\R_+^*)^\A$ to $Z$, with $\A$ finite.
    Let $(Z_a)_{a \in \A}$ be the partition of $Z$ such that elements of $Z_a$ have return time $h_a$.
    Coding of orbits define a subshift $(X,S)$ over alphabet $\A$.
    Assume that the coding map $\cod : Z \to X$ is measurable, that the subshift is minimal,
    and that $\pi(\L_X)$ is bounded, where $\pi = \lambda h + w$, with row vectors $w \in \Z^\A$ and $h = (h_a)_{a \in \A}$.
    Then, $\lambda$ is an eigenvalue of the $\R$-action.
\end{prop}

\begin{proof}
    As in the proof of Proposition~\ref{prop:gs}, we define a continuous map $f : X \to \R/\Z$ such that
    $\forall x \in X$, $f(S^n x) = \pi(x_{[0,n)}) + f(x) = \lambda h \ab(x_{[0,n)}) + f(x)$. 
    Then, $g = f \circ \cod$ is a measurable map such that
    $\forall z \in Z$, $\forall t \in \R$ such that $t.z \in Z$, $g(t.z) = \lambda t + g(z)$.
    Then, the map $h:Y \to \R/\Z$ defined by $h(t.z) = \lambda t + g(z)$ for every $t \in \R$ and $z \in Z$
    is a measurable eigenfunction associated with $\lambda$. 
\end{proof}

\section{Eigenvalues of $\S$-adic subshifts} \label{sec:Sadic}

This section aims to characterize continuous eigenvalues of $\S$-adic subshift using coboundaries.

\subsection{Main results}
The two following theorems partially extends the result of Host for a single substitution (see~\cite{Host86} Theorem~(1.4)).



\begin{thm} \label{thm:nec}
    Let $\bm \sigma$ be a finitary primitive directive sequence. 
    If $t \in \R/\Z$ is a continuous eigenvalue of $(X_{\bm \sigma}, S)$,
    then there exists $k \in \N$, 
    and a linear form $c : \R^{\A_k} \to \R$ such that $c(\ab w) \in \Z$ for every return word $w$ of $\L^{k}$ and
    such that
    \[
        (t (1, ..., 1) M_{[0, k)} + c)M_{[k, n)} \xrightarrow{n \to \infty} 0.
    \]
\end{thm}



Note that such a map $c$ give a coboundary of $X_{S^k \bm \sigma}$.
If we remove the finitary hypothesis and replace it by $\liminf_{n \to \infty} \abs{\A_n} < \infty$, then we still have the conclusion but with a convergence modulo $1$.

If the speed of convergence is fast enough, we have the reciprocal,
which is a generalization of Theorem~{4.8} and Theorem~{4.12} in~\cite{BBY22}.

\begin{thm} \label{thm:suf}
    Let $\bm \sigma$ be a finitary, primitive and recognizable directive sequence.
    If there exists $k \in \N$ and a linear form $c : \R^{\A_{k}} \to \R$ such that $c(\ab w) \in \Z$ for every return word $w$ of $\L^{k}$ and
    such that
    \[
        \sum_{n \geq k} (t (1, ..., 1) M_{[0,k)} + c)M_{[k, n)}
    \]
    converges, then $t$ is a continuous eigenvalue of $(X_{\bm \sigma},S)$.
\end{thm}


\begin{rem} \label{rem:unimod}
    In Theorem~\ref{thm:suf}, if the substitution $\sigma_{[0,k)}$ is unimodular,
    then the conclusion of the theorem holds without the recognizability hypothesis.
    Indeed, $c$ can be decomposed as the sum of a real coboundary $\tilde{c}$ and an integer row vector $w$ by Lemma~\ref{lem:rep}, and $\tilde{c} M_{[0,k)^{-1}}(\L_X)$ is bounded and $w M_{[0,k)}^{-1}$ is an integer row vector, thus hypotheses of Corollary~\ref{cor:Sadic} are satisfied with $\pi = t(1...1) + w M_{[0,k)}^{-1}$.
\end{rem}

\begin{proof}[Proof of Theorem~\ref{thm:suf}]
    Since $\bm \sigma$ is recognizable, the maps $p_n : X_{\bm \sigma} \to \L^{(n)}$ are well-defined and continuous.
    Then, we define a continuous map $f : X_{\bm \sigma} \to \R/\Z$ by
    \[
        f = \sum_{n=0}^{k-1} t (1 ... 1) M_{[0,n)} \ab p_n + \sum_{n \geq k} ( t (1 ... 1) M_{[0,k)} + c) \ab \sigma_{[k,n)} p_n + c (\theta_k),
    \]
    where the last occurrence of $c$ is a map $c : X_{S^k \bm \sigma} \to \R/\Z$ such that $c(Sx) + c(x_0) = c(x)$,
    and $\theta_k : X_{\bm \sigma} \to X_{S^k \bm \sigma}$ is the map such that $\sigma_{[0,k)} (\theta_k(x)) = \sigma_{[0,k-1)}(p_{k-1}(x)) ... \sigma_{0}(p_1(x)) p_0(x) x$.
    In other words, $\theta_k(x)$ is the word obtained after $k$ steps of desubstitution from $x$.
    The map $f$ is continuous since the series is normally convergent.
    Let us show that $f$ is an eigenfunction for the eigenvalue $t$.
    Let $x \in X_{\bm \sigma}$ be such that $p_n(x) = \epsilon$ for every $n \in \N$.
    Since $\bm \sigma$ is primitive, $(X_{\bm \sigma}, S)$ is minimal, and
    it is enough to prove the equality $f \circ S = f + t$ on the orbit of $x$.
    Let $m \in \N$. There exists $N \geq k$ such that $p_n(S^m x) = \epsilon$ for every $n \geq N$,
    thus 
    \begin{eqnarray*}
        f(S^m x) &=& \sum_{n=0}^k t (1 ... 1) \ab \sigma_{[0,n)} \ab p_n(S^m x) + \\
                & & \quad \sum_{n = k}^N ( t (1 ... 1) \ab(\sigma_{[0,k)}) + c) \ab \sigma_{[k,n)} p_n(S^m x) + c \theta_{k} (S^m x) \\
                 &=& t (1 ... 1) \ab(\sigma_{[0,N)}(p_N)...\sigma_0(p_1)p_0) + \\
                & &  \quad c \ab(\sigma_{[k,N)}(p_N)...\sigma_k(p_{k+1})p_k) + c\theta_{k} (S^m x) \\
                 &=& m t + c (\theta_k (x)) = f(x) + mt. 
    \end{eqnarray*}
    This before last equality comes from the relations
    \[
        \sigma_{[0,N)}(p_N(S^m x)) ... \sigma_0(p_1(S^m x)) p_0 (S^m x) S^m x = x,
    \]
    \[
        \sigma_{[k,N)}(p_N(S^m x)) ... \sigma_k(p_{k+1}(S^m x)) p_k (S^m x) \theta_k(S^m x) = \theta_k(x).
    \]
\end{proof}

To prove Theorem~\ref{thm:nec}, we need to introduce some important notions such as nested sequences.

\subsection{Nested sequences} \label{ss:nested}

In this subsection, we consider subshifts over infinite words.
Let $\bm \sigma$ be a directive sequence over alphabets $(\A_n)$.
In this subsection we consider only first letters in images of letters,
so we can assume that substitutions of $\bm \sigma$ have constant lengths $1$.
A \defi{nested sequence} is a sequence $(a_n)_{n \in \N} \in \Pi_{n \in \N} \A_n$ of letters such that
$\forall n \in \N,\ \sigma_n(a_{n+1}) = a_n$.
We denote by $\NN_{\bm \sigma}$ (or just $\NN$) the set of nested sequences. 

\begin{lem} \label{lem:N:non-empty}
    If alphabets $\A_n$ are non-empty, then $\NN \neq \emptyset$.
\end{lem}

\begin{proof}
    We assume that for all $n \in \N$, $\sigma_n$ has length one, up to replace $\sigma_n(a)$ by its first letter for every $a \in \A_n$.
    The intersection $\bigcap_{n \in \N} \sigma_{[0,n)}(\A_n)$ is non-empty since $(\sigma_{[0,n)}(\A_n))$ is a decreasing sequence of non-empty subsets of $\A_0$. We then construct a nested sequence $(a_n)$ by induction: if $a_n \in \bigcap_{k \geq n} \sigma_{[n,k)}(\A_k)$, then we can find $a_{n+1} \in \bigcap_{k \geq n+1} \sigma_{[n+1,k)}(\A_k)$ such that $\sigma_n(a_{n+1}) = a_n$.
\end{proof}

The following tells us that nested sequences are almost pairwise disjoint. 
\begin{lem} \label{lem:N:disj}
    If $(a_n)$ and $(b_n)$ are two different nested sequences,
    then
    \[
        \exists n_0 \in \N,\ \forall n \geq n_0,\ a_n \neq b_n.
    \]
\end{lem}
\begin{proof}
    If $a_{n_0} = b_{n_0}$, then for every $n \leq n_0$, $a_n = b_n$. Thus, if the conclusion was not true, then the two nested sequences would be equal.
\end{proof}

We deduce the following.

\begin{lem} \label{lem:N:finite}
    \[
        \abs{\NN} \leq \liminf_{n \to \infty} \abs{\A_n}.
    \]
\end{lem}



We assume in the following that alphabets $\A_n$ are disjoints.
For each $k \in \N$, let $\B_k = \{a \in \A_k \mid \forall (a_n) \in \NN,\ \sigma_{[0,k)}(a) \not\in [a_0] \}$.
This set corresponds to letters that never reach a nested sequence.
We have the following.

\begin{lem}
    \[
        \exists n_0 \in \N,\ \forall n \geq n_0,\ \B_n = \emptyset.
    \]
\end{lem}

\begin{proof}
    If $a \in \B_{n+1}$, then the first letter of $\sigma_n(a)$ is in $\B_n$.
    Thus, we can define $\tau_n : \B_{n+1} \to \B_n$ such that $\tau_n(a)$ is the first letter of $\sigma_n(a)$.
    By definition of the sets $\B_n$, $(\tau_n)$ has no nested sequence.
    Thus, by Lemma~\ref{lem:N:non-empty}, $\exists n_0 \in \N,\ \B_{n_0} = \emptyset$.
    Then, for every $n \geq n_0$, $\B_n = \emptyset$ since $\tau_{[n_0,n)}(\B_{n}) \subseteq \B_{n_0}$.
\end{proof}

We deduce the following.

\begin{lem} \label{lem:alpha}
    There exists $n_0 \in \N$ and a map $\alpha : \bigcup_{n \geq n_0} \A_n \to \NN$
    such that $\forall n \geq n_0,\ \forall a \in \A_n,\ \sigma_{[0,n)}(a) \in [\alpha(a)_0]$.
\end{lem}

\subsection{Proof of Theorem~\ref{thm:nec}}

Let $\bm \sigma$ be a primitive directive sequence with $\liminf_{n \to \infty} \abs{\A_n} < \infty$.
Thanks to Proposition~\ref{prop:inf:bi-inf}, we can work only with right-infinite subshifts as in previous subsection.

\begin{lem}
    For every nested sequence $\bm a = (a_n) \in \NN$, $(\sigma_{[0,n)}(a_n))$ is a nested sequence of words that converges to an infinite word $w_{\bm a}$.
\end{lem}


Let $f$ be a continuous eigenfunction associated with an eigenvalue $t \in \R/\Z$.
Let
\[
    \F = \left\{f(w_{\bm a}) \in \R/\Z \mid \bm a \in \NN \right\}.
\]
By Lemma~\ref{lem:N:finite}, the set $\F$ is finite.
Let $\epsilon > 0$ such that $\epsilon < \min \{ \abs{x - y} \mid x \neq y \in \F \}$.
Let $n_1 \in \N$ be such that
\[
    \forall \bm a \in \NN,\ \abs{ f \left( \left[ \sigma_{[0,n_1)}(a_{n_1}) \right] \right) -
                                f(w_{\bm a}) } \leq \epsilon/4.
\]
Then, let $n_0 \in \N$ and $\alpha$ be given by Lemma~\ref{lem:alpha} for the directive sequence $S^{n_1} \bm \sigma$.
We have
\begin{eqnarray} \label{eqn}
    \forall n \geq n_0,\ \forall x \in X_{\bm \sigma}^{(n)},\ \abs{ f(\sigma_{[0,n)}(x)) - v_{x_0} } \leq \epsilon/4,
\end{eqnarray}
where we denote $v_a = f(w_{\alpha(a)})$.

\begin{lem} \label{lem:ab}
    \[
        \forall n \geq n_0,\ \forall ab \in \L^{(n)},\ \abs{t h_n(a) + v_a - v_b } \leq \epsilon/2,
    \]
    where $h_n(a) = \abs{\sigma_{[0,n)}(a)} = (1 ... 1) M_{[0,n)} e_a$.
\end{lem}

\begin{proof}
    For every $n \geq n_0$, for every $ab \in \L_{\bm \sigma}^{(n)}$ and for every $u \in [ab]$, we have
    \[
        t h_n(a) + f(\sigma_{[0,n)}(u)) = f(S^{h_n(a)} \sigma_{[0,n)}(u)) = f(\sigma_{[0,n)}(S u)).
    \]
    We deduce the result by triangular inequality, using~(\ref{eqn}).
\end{proof}

We define $c_{n} : \R^\A \to \R/\Z$ by
$c_n(a) = v_a - v_b$ if $ab \in \L^{(n)}$.
It is well defined. Indeed, if $ab \in \L^{(n)}$ and $ac \in \L^{(n)}$ then by Lemma~\ref{lem:ab} and by triangular inequality,
we have $\abs{v_c - v_b} \leq \epsilon < \min \{ \abs{x - y} \mid x \neq y \in \F \}$, thus $v_b = v_c$.
And if $w$ is a return word of $\L_n$, then clearly $c_n(w) = 0$ since we have $wa \in \L^{(n)} \Rightarrow c_n(w) = v_{w_0} - v_a$. 
Thus, $c_n$ if a coboundary of $X_{\bm \sigma}^{(n)}$.

\begin{lem} \label{lem:cn}
    For every $n \geq n_0$, we have $c_{n+1} = c_n \circ \sigma_n$.
\end{lem}

\begin{proof}
    Let $ab \in \L^{(n+1)}$.
    We have $c_{n+1}(a) = v_a - v_b$.
    Let $a_0$ and $b_0 \in \L^{(n)}$ be respectively the first letters of $\sigma_n(a)$ and $\sigma_n(b)$.
    We have $\sigma_n(a)b_0 \in \L^{(n)}$, thus $c_n(\sigma_n(a)) = v_{a_0} - v_{b_0}$.
    Now, by definition we have $v_a = v_{a_0}$ and $v_b = v_{b_0}$ thus $c_{n+1}(a) = c_n (\sigma_n(a))$.
\end{proof}

Since $\epsilon > 0$ could be chosen arbitrarily small, Lemma~\ref{lem:ab} gives
\[
    t (1 ... 1)M_{[0,n)} + c_n \xrightarrow{n \to \infty} 0 \text{ mod } 1.
\]
Thus, by Lemma~\ref{lem:cn} we get $(t (1 ...1)M_{[0,n_0)} + c_{n_0}) M_{[n_0, n)} \xrightarrow{n \to \infty} 0 \text{ mod } 1$.

Now we need the following easy generalization of Lemma~1 in~\cite{Host86}.
It is where we use the finitary hypothesis.

\begin{lem} \label{lem:mats}
    Let $(M_n)_{n \in \N} \in M_{d_n, d_{n+1}}(\Z)$ with bounded coefficients and with $\forall n \in \N, d_n \geq 1$. 
    If $v \in \R^{d_0}$ is such that $v M_{[0,n)} \xrightarrow{n \to \infty} 0 \text{ mod } 1$.
    Then there exists $k \in \N$ and $w \in \Z^{d_k}$ such that 
    $(v M_{[0,k)} - w)M_{[k,n)} \xrightarrow{n \to \infty} 0$.
\end{lem}

\begin{proof}
    For every $n \in \N$, let $\epsilon_n \in (-1/2,1/2]^{d_n}$ and $w_n \in \Z^{d_n}$ be such that $v M_{[0,n)} = \epsilon_n + w_n$.
    We have $w_n M_n - w_{n+1} = \epsilon_{n+1} - \epsilon_n M_n \xrightarrow{n \to \infty} 0$,
    and $w_n M_n - w_{n+1}$ has integer coefficients.
    Thus there exists $k \in \N$ such that $\forall n \geq k,\ w_n = w_{n+1} M_n$.
    Then, $(v M_{[0,k)} - w_k) M_{[k,n)} \xrightarrow{n \to \infty} 0$.
\end{proof}

Let $\tilde c_{n_0} : \R^{\A_{n_0}} \to \R$ be any representative of $c_{n_0}$.
We use this lemma with matrices of the directive sequence $S^{n_0} \sigma$
and with the vector $t (1...1)M_{[0,n_0)} + \tilde c_{n_0}$.
Then, it gives a $k \geq n_0$ and a vector $w \in \Z^{\A_k}$ such that
\[
    ((t (1 ...1)M_{[0,n_0)} + \tilde c_{n_0} )M_{[n_0, k)} - w)M_{[k,n)} \xrightarrow{n \to \infty} 0.
\]
Let $c = \tilde c_{n_0} M_{[n_0,k)} - w$. Then, we have
\[
    (t (1 ...1)M_{[0,k)} + c)M_{[k,n)} \xrightarrow{n \to \infty} 0.
\]
Let $w$ be a return word of $\L^{\A_k}$, then we clearly have $c(w) \in \Z$ since $c = c_k$ mod $\Z^{\A_k}$.
This ends the proof of Theorem~\ref{thm:nec}.

\section{Eigenvalues of morphic subshifts} \label{sec:morphic}

Recall that a \defi{morphic} subshift is the subshift of a directive sequence of the form $\tau \sigma^\omega$.
In other words, it is the subshift generated by $\tau(X_\sigma)$.

\subsection{Characterization of eigenvalues of morphic subshifts}

The following result allows us to compute eigenvalues of morphic subshifts.

\begin{thm} \label{thm:car:eigs}
    Let $\sigma$ be a primitive aperiodic substitution, and let $\tau$ be a word morphism recognizable in $X_\sigma$.
    Then, the set of eigenvalues of $(X_{\tau \sigma^\omega}, S)$ is
    \[
        \left( W \cap \Delta_{M_\sigma} \right) v
    \]
    where
    \begin{itemize}
        \item $v$ is a generalized eigenvector of $M_\sigma$ for an eigenvalue of modulus $\geq 1$, annihilated by coboundaries (for example the Perron eigenvector) and normalized such that $(1...1)M_\tau v = 1$,
        \item $\Delta_{M_\sigma} = \{w \in \Q^{\A_\sigma} \mid \exists n \in \N,\ w M_{\sigma}^n \in \Z^{\A_\sigma}\}$, \\
        can be replaced by $\bigcup_{n \in \N} \Z^{\A_\sigma} M_\sigma^{-1}$ where $M_\sigma^{-1}$ is the pseudo-inverse of $M_\sigma$, defined in Subsection~\ref{ss:matrix},
        \item $W = \{w \in \Q^{\A_\sigma} \mid \exists x \in \R^d,\ xCV = w V\}
                = \Q^{\A_\sigma} \cap (C^\circ \cap \Span(V))^\circ$ is a $\Q$-vector space
        where
            \begin{itemize}
                \item $V$ is a matrix whose columns are bases of generalized eigenspaces for eigenvalues of modulus $\geq 1$,
                \item $C$ is a matrix whose first row is $(1...1)M_\tau$ and other rows are a basis of the set of coboundaries $R^\circ$ (see definition in Subsection~\ref{ss:cob}). 
            \end{itemize}
    \end{itemize}
\end{thm}

The set $\Delta_{M_\sigma}$ is linked with another invariant called dimension group, see~\cite{DurandPerrin2022} for more details.

Note that if we remove the recognizability or aperiodicity hypothesis, then the set of eigenvalues
is between $(W \cap \Z^{\A_\sigma}) v$ and $(W \cap \Delta_{M_\sigma}) v$.

We show that eigenvalues of morphic subshifts are continuous.

\begin{prop} \label{prop:car}
    Under hypotheses of the previous theorem,
    the following statements are equivalent
    \begin{itemize}
        \item $t \in \R/\Z$ is a measurable eigenvalue of $(X_{\tau \sigma^\omega}, S)$. 
        \item $t \in \R/\Z$ is a continuous eigenvalue of $(X_{\tau \sigma^\omega}, S)$.
        \item There exists an integer $k \in \N$, a real coboundary $c$ of $X_\sigma$,
    and a row vector $w \in \Z^{\A_\sigma}$ such that for every generalized eigenvector $v$ of $M_\sigma$ associated with an eigenvalue of modulus $\geq 1$,
    \[
        (t(1...1)M_\tau M_\sigma^k + c - w)v = 0.
    \]
    \end{itemize}
\end{prop}

\begin{proof}[Proof of Proposition~\ref{prop:car}]
    By Theorem~\ref{thm:mosse}, $\sigma$ is recognizable in the bi-infinite subshift $X_\sigma$.
    Thus, the directive sequence $\tau \sigma^\omega$ is recognizable and it is primitive.
    The last point gives a map $c : \R^{\A_\sigma} \to \R$ such that $c(r) \in \Z$ for every return word $r$ of $X_\sigma$ and
    such that $(t(1...1)M_\tau M_\sigma^k +c)M_\sigma^n \xrightarrow{n \to \infty} 0$.
    Note that the convergence is exponential.
    Thus, by Theorem~\ref{thm:suf}, this implies that $t$ is a continuous eigenvalue of $(X_{\tau \sigma^\omega},S)$.
    
    Obviously, a continuous eigenvalue is a measurable eigenvalue.
    
    Let $t \in \R/\Z$ be a measurable eigenvalue.
    We can describe the shift $(X_{\tau \sigma^\omega},S)$ with another directive sequence which is left-proper,
    using the notion of return substitution, defined in Subsection~\ref{ss:rs} thereafter.
    Indeed, we can assume that there exists $a \in \A_\sigma$ such that $\sigma(a) \in [a]$ even if its means replacing $\sigma$ by a power of itself.
    Then, a power of the return substitution $\partial_a \sigma$ is left-proper and there exists $\theta : \R \to \A_\sigma^*$ such that $\theta \circ \partial_a \sigma = \sigma \circ \theta$.
    Thus, the directive sequence $\tau \theta (\partial_a \sigma)^\omega$ has the same subshift as $\tau \sigma^\omega$.
    It implies by Theorem~7.3.5 in~\cite{DurandPerrin2022} (see also~\cite{Durand2000} and~\cite{Durand2003}) that the shift is linearly recurrent.
    Then, thanks to Proposition~14 in~\cite{CDHM2003} the eigenvalue $t$ is continuous.
    
    Let $t \in \R/\Z$ be a continuous eigenvalue of $(X_{\tau \sigma^\omega},S)$.
    Then, by Theorem~\ref{thm:nec}, there exists $k \in \N$ and a map $c : \R^{\A_\sigma} \to \R$ such that
    $c(r) \in \Z$ for every return word $r$ of $X_\sigma$ and such that
    \[
        (t (1...1) M_\tau M_\sigma^k + c) M_{\sigma}^n \xrightarrow{n \to \infty} 0.
    \]
    Then, by Lemma~\ref{lem:rep}, $c$ can be decomposed as the difference of a coboundary $c' \in (\R^{\A_\sigma})^*$ and an integer row vector $w \in \Z^{\A_\sigma}$.
    And for every row vector $u$ we have $u M_\sigma^n \xrightarrow{n \to \infty} 0 \Longleftrightarrow u V = 0$
    where $V$ is a matrix whose columns are generalized eigenvectors for eigenvalues of modulus $\geq 1$.
    Hence, we obtain the third point.
\end{proof}

\begin{proof}[Proof of Theorem~\ref{thm:car:eigs}]
    The set $W$ is indeed a $\Q$-vector subspace of $\Q^{\A_\sigma}$ since we obtain it as the kernel of some matrix obtained by making the system $x CV = w V$ in echelon form (see Subsection~\ref{ss:algo} for more details).
    
    Let us show that we can replace $\Delta_{M_\sigma}$ as claimed.
    Let $I = \bigcap_{n \in \N} \Q^{\A_\sigma} M_\sigma^n$ be the eventual left-image of $M_\sigma$ and let $E_0 = \bigcup_{n \in \N} \ker_{\operatorname{left}} M_\sigma^n$ be its generalized left-eigenspace for the eigenvalue $0$.
    We have $I + E_0 = \Q^{\A_\sigma}$.
    The set $E_0$ is included in both $W$ and $\Delta_M$, and $W \cap \Delta_M$ is an additive group.
    Thus, we have $E_0 + (W \cap \Delta_{M_\sigma} \cap I) = W \cap \Delta_{M_\sigma}$.
    The set $I$ is invariant by left multiplication by $M_\sigma$, and the restriction of $M_\sigma$ to $I$ is invertible,
    so we have $\Delta_M \cap I = \bigcup_{n \in \N} \Z^{\A_\sigma} M_\sigma^{-n} \cap I$.
    Thus, $(W \cap \Delta_{M_\sigma}) v = (W \cap \Delta_{M_\sigma} \cap I) v = (W \cap \bigcup_{n \in \N} \Z^{\A_\sigma} M_\sigma^{-n}) v$.
    
    Let $w \in W \cap \Delta_{M_\sigma}$.
    Let us prove that $w v$ is an eigenvalue of $(X_{\tau \sigma^\omega}, S)$.
    By definition of $W$, there exists $x \in \R^d$ such that $x CV = w V$.
    By definition of $C$, $x C$ is of the form $t(1...1) M_\tau + c$ for some real coboundary $c \in R^\circ$ of $X_\sigma$.
    By definition of $\Delta_{M_\sigma}$, there exists $k \in \N$ such that $w M_\sigma^k \in \Z^{\A_\sigma}$.
    Then $(t(1...1)M_\tau + c - w)V = 0$, and since generalized eigenspaces are invariant by $M_\sigma$,
    we have $(t(1...1)M_\tau M_\sigma^k + c' - w') V = 0$,
    where $c' = c M_\sigma^k \in R^\circ$ is a real coboundary of $X_\sigma$ and $w' = w M_\sigma^k \in \Z^{\A_\sigma}$.
    By Proposition~\ref{prop:car}, $t$ is an eigenvalue of $(X_{\tau \sigma^\omega}, S)$.
    Furthermore, we have $w v = (t(1...1)M_\tau +c)v = t(1...1) M_\tau v = t$.
    
    Reciprocally, let $t$ be an eigenvalue.
    By Proposition~\ref{prop:car}, there exists a real coboundary $c \in R^\circ$, $k \in \N$ and $w \in \Z^{\A_\sigma}$ such that
    $(t(1...1)M_\tau M_\sigma^k + c - w)V = 0$.
    Then, since generalized eigenspaces are invariant by $M_\sigma^{-1}$, we have $(t(1...1)M_\tau + c' - w') V = 0$ where $c' = c M_\sigma^{-k}$ and $w' = w M_\sigma^{-k}$,
    where $M_\sigma^{-1}$ is the pseudo-inverse of $M_\sigma$.
    We have $w' \in \bigcup_{n \in \N} \Z^{\A_\sigma} M_\sigma^{-n}$.
    Now, we have $M_\sigma R \subseteq R$ since the image of a return word by $\sigma$ is a product of return words.
    So $(R^\circ \cap I) M_\sigma \subseteq R^\circ \cap I$ where $I = \bigcap_{n \in \N} \Q^{\A_\sigma} M_\sigma^n$.
    Since the restriction of $M_\sigma$ to $I$ is invertible, and $R^\circ M_\sigma^n \subseteq R^\circ \cap I$ for $n$ large enough, we get $(R^\circ \cap I) M_\sigma = R^\circ \cap I$.
    Hence, $c'V \in R^\circ V$.
    Thus, there exists $x$ such that $xCV = w'V$.
    We deduce that $w' \in W$.
    Moreover we have $w'v = (t(1...1)M_\tau + c')v = t(1...1)M_\tau v = t$.
\end{proof}

\subsection{The algorithm} \label{ss:algo}


Let $\sigma$ and $\tau$ be two substitutions with $\sigma$ primitive aperiodic and $\tau$ recognizable in $X_\sigma$.
The algorithm to compute eigenvalues of $(X_{\tau \sigma^\omega}, S)$ is the following.

\begin{itemize}
    \item Compute the coboundary space $R^\circ$ of $X_\sigma$. See Subsection~\ref{ss:compute:cob} for more details.
    \item Let $C$ be the matrix with first row $(1...1)M_\tau$ and other rows with a basis of the coboundary space $R^\circ$.
        Let $d$ be the number of rows of $C$.
    \item Compute a matrix $V$ whose columns are bases of generalized eigenspaces of $M_\sigma$
            for eigenvalues of modulus $\geq 1$, with coordinates in the algebraic field $\overline{\Q}$.
    \item Echelon the system $x CV = w V$ with unknowns $x \in \R^{d}$, where $w$ is a row vector.
            It gives a matrix whose columns are linear combination of columns of $V$ and whose left-kernel
            intersected with $\Q^{\A_\sigma}$ gives $W$.
            Decompose each columns in a basis of a common number field of coefficients.
            It gives a rational matrix whose left-kernel is $W$.
            Then we can compute a basis of the $\Q$-vector space $W$.
    
    \item Compute the Perron eigenvector $v$ of $M_\sigma$ such that $(1...1)M_\tau v = 1$.
    
    \item Then, the set of eigenvalues is
        \[
            (W \cap \bigcup_{n \in \N} \Z^{\A_\sigma} M_\sigma^{-k}) v
        \]
        where $M_\sigma^{-1}$ is the pseudo-inverse of $M_\sigma$.
\end{itemize}

Note that we may replace the Perron eigenvector $v$ by any generalized eigenvector annihilated by coboundaries and such that
$(1...1)M_\tau v = 1$.
Note that in general, the set $(W \cap \Delta_{M_\sigma}) v$ is not equal to $W v \cap \Delta_{M_\sigma} v$ since $W$ has no reason to be $M_\sigma$-invariant.
However, in many situations, the set of eigenvalues can be made more explicit:

\begin{itemize}
    \item If $\sigma$ is pseudo-unimodular (i.e. the product of non-zero eigenvalues is $\pm 1$), 
        then the set of eigenvalues is $(W \cap \Z^{\A_\sigma}) v$.
        
        Indeed, by Lemma~6.4 in~\cite{Me23}, if $M_\sigma$ is pseudo-unimodular then $\Delta_{M_\sigma}$ is equal to 
        $\Z^{\A_\sigma}$ modulo the generalized eigenspace for the eigenvalue $0$.
        
    \item If the matrix $CV$ is invertible, or more generally if $W M_\sigma = W$,
        and if $v$ is an eigenvector for an eigenvalue $\beta$,
        then the set of eigenvalues is
        \[
            \Z \left[ \frac{1}{\beta} \right] (W \cap \Z^{\A_\sigma}) v.
        \]
        Indeed, we have
        \[
            \left( W \cap \bigcup_{n \in \N} \Z^{\A_\sigma} M_\sigma^{-n} \right) v
                        = \bigcup_{n \in \N} (W M_\sigma^n \cap \Z^{\A_\sigma}) v/\beta^{n}.
        \]
        
    \item If $(1...1)M_\tau$ is a left-eigenvector of $M_\sigma$ for an eigenvalue $\beta$, then 
        we fall into the previous case. Indeed, it gives that $C$ and $C M_\sigma$ have same row space
        modulo the generalized left-eigenspace of $M_\sigma$ for the eigenvalue $0$ and we have
        $xCV = wV \Leftrightarrow (xC - w)M_\sigma V = 0 \Leftrightarrow xCM_\sigma V = w M_\sigma V$, so $W M_\sigma = W$.
        
        Note that when $\tau = \id_{\A_\sigma}$, it corresponds to the constant length case.

    \item If there exists $n \in \N$ such that $M_\sigma^n = 0$ mod $\det(M_\sigma)$.
        Then, the set of eigenvalues is
        \[
            \Z \left[ \frac{1}{\det(M_\sigma)} \right] (W \cap \Z^{\A_\sigma}) v.
        \]
        Indeed, under this hypothesis we have $\Delta_M = \bigcup_{n \in \N} \frac{1}{\det(M_\sigma)^n} \Z^{\A_\sigma}$.
\end{itemize}

See Section~\ref{sec:exs} for examples of computations of eigenvalues.
Algorithms of this paper have been implemented in the package eigenmorphic
(see~\url{https://gitlab.com/mercatp/eigenmorphic})
for the Sage math software (see~\url{https://www.sagemath.org/}).
It can be installed with the command line
\begin{lstlisting}[language=bash]
  $ sage -pip install eigenmorphic
\end{lstlisting}
Examples of computations with this package can be found here:
\url{https://www.i2m.univ-amu.fr/perso/paul.mercat/SubshiftsEigenvalues.pdf}

\subsection{Computation of coboundaries} \label{ss:compute:cob}

The following lemma permits to reduce the computation of the return space to the one of return space on only one letter,
if we want to compute coboundaries only modulo the generalized left eigenspace $E_0$ for the eigenvalue $0$.
It is indeed what we need in the computation of eigenvalues of morphic subshifts since $E_0 V = 0$.

\begin{lem}
    Let $\sigma$ be a primitive substitution over alphabet $\A$ and let $a \in \A$.
    Let $R$ be the return space and $R_a$ be the return space on letter $a$.
    Let $I = \bigcap_{n \in \N} M_\sigma^n \Q^\A$.
    Then
    \[
        I \cap R = I \cap R_a \quad \text{ and }\quad
        R^\circ + E_0 = R_a^\circ + E_0,
    \]
    where $E_0 = I^\circ$ is the generalized left-eigenspace of $M_\sigma$ for the eigenvalue $0$.
\end{lem}

\begin{proof}
    The image of a return word by a substitution is a product of return words. Thus $M_\sigma R \subseteq R$.
    Since $\restriction{M_\sigma}{I}$ is invertible, we have $M_\sigma (R \cap I) = R \cap I$.
    Let $k \in \N$ such that $M_\sigma^k > 0$.
    Let $(w)_{w \in J}$ be a family of return words such that $(\ab w)_{w \in J}$ is a basis of $R \cap I$.
    Then, $(M_\sigma^k \ab w)_{w \in J}$ is also a basis of $R \cap I$.
    Let $w \in J$. Let $p,s \in \A^*$ such that $\sigma^k(w_0) = pas$.
    The word $\sigma^k(w w_0)$ is in the language of $\sigma$, so $as\sigma^k(Sw)p$ is a product of return word on letter $a$.
    Moreover, $\ab(as\sigma^k(Sw)p) = \ab(\sigma^k(w))$.
    Thus $R \cap I \subseteq R_a \cap I$. The other inclusion is obvious.
    Then, we have $R^\circ + E_0 = (R \cap I)^\circ = (R_a \cap I)^\circ = R_a^\circ + E_0$.
\end{proof}


The computation of the set $\RR_a$ of return words on letter $a$ can be done by computing a substitution
called return substitution and whose alphabet is the set of return words on letter $a$.
This is done in the next subsection.

\begin{rem}
    The coboundary space can be computed more efficiently by computing the coboundary graph,
    that is the largest invariant oriented graph
    with edges labeled by letters where each letter appears exactly once.
\end{rem}

\subsection{Return substitution} \label{ss:rs}

The notion of return substitution comes from F. Durand (see~\cite{Durand98}).
We recall the definition and how to compute it even if it is well-know, in order to have a complete algorithm.
Let $\sigma$ be a primitive substitution, and let $a \in \A_\sigma$ be a letter such that $\sigma(a) = a$.
Note that such a letter always exists up to replacing $\sigma$ with a power of itself.
The \defi{return substitution} on letter $a$ is denoted as $\sigma_a$ or $\partial_a \sigma$.
Its alphabet is the set $\RR_a$ of return words on letter $a$, and its satisfies
$\sigma_a \circ \varphi_a = \varphi_a \circ \sigma$, where $\varphi_a : \RR_a \to \A_\sigma^*$
maps an element of $\RR_a$ to the corresponding word.
We can also extend this definition in an obvious manner:
if $w$ is a prefix of a fixed point of $\sigma$, then we the define the return substitution $\sigma_w$ accordingly.

We can compute the return substitution with the following algorithm:
\begin{itemize}
    \item Compute $\sigma^n(a)$ from $n=1$ and up to find a first return word on letter $a$.
    \item Take the image of a return word $r$ by $\sigma$: it is a product of return words.
            We deduce what is $\sigma_a(r)$.
    \item Continue the previous step with every new return word found.
\end{itemize}
When this algorithm stops, then we have found every return word on letter $a$.
Indeed, if $x$ is the fixed point of $\sigma$ starting by letter $a$
and if $A$ is such that $\varphi_a(A)$ is the first return word found,
then $x$ is equal to $\varphi_a(y)$ where $y$ is the fixed point of $\sigma_a$ starting by letter $A$.
Since $\sigma$ is primitive, every return word appears in $x$, thus the alphabet of $\sigma_a$
is the set of return words on letter $a$.

\begin{ex}
    Consider the Thue-Morse substitution 
    \[
        \sigma: a \mapsto ab,\ b \mapsto ba.
    \]
    If we compute $\sigma^2(a) = abba$, we see that $abb$ is a return word.
    Then, we compute $\sigma(abb) = abbaba$. We found two new return words: $\{ab, a\}$.
    We compute $\sigma(ab) = abba$ and $\sigma(a) = ab$. There are no more return words.
    Thus the return substitution is
    \[
        \sigma_a : A \mapsto ABC,\ B \mapsto AC,\ C \mapsto B
    \]
    with $\varphi_a : A \mapsto abb,\ B \mapsto ab,\ C \mapsto a$.
\end{ex}

\subsection{Deciding aperiodicity} \label{ss:period}

Let $\sigma$ be a primitive substitution.
Using return substitutions, we can provide an algorithm to decide if $(X_\sigma, S)$ is periodic.
The fact that aperiodicity is decidable is already well-know, but we give an algorithm for completeness.

Assume that $\sigma(a) = a$, up to replacing $\sigma$ by a power of itself.
We define the $n$-th derivative of $\sigma$ for the letter $a$ by induction on $n \in \N$:
if $n = 0$, $\partial_a^n \sigma = \sigma$; and if $n > 0$, $\partial_a^{n} = \partial \partial_a^{n-1} \sigma$,
where the return substitution is taken for the first letter, ensuring that the fixed point with this first letter is mapped
to the fixed point of the original substitution starting with $a$.

\begin{lem}
    $(X_\sigma, S)$ is periodic if and only if there exists $n \in \N$ such that the alphabet of $\partial_a^n \sigma$
    has only one letter.
\end{lem}

\begin{proof}
    Since $\sigma$ is primitive, the shift is periodic if and only if the fixed point $x$ starting by letter $a$
    is periodic.
    And a word is periodic if and only if the sequence of its derivatives (i.e. rewriting the word as product of its return words on first letter) becomes constant to the word with only one letter.
\end{proof}

Such criterion can be algorithmically tested, since the sequence $\partial_a^n \sigma$ is ultimately periodic
thank to the following proposition (see Proposition~7.6.7 in~\cite{DurandPerrin2022}), and since $\partial^n \sigma$ is of the form $\sigma_p$ for some prefix $p$ of a fixed point of $\sigma$.

\begin{prop}
    Let $x$ be a fixed point of a primitive substitution $\sigma$.
    Then the set $\{\sigma_w \mid w \text{ prefix of } x\}$ is finite.
\end{prop}

\begin{ex}
    Consider the Thue-Morse substitution
    \[
        \sigma: a \mapsto ab,\ b \mapsto ba.
    \]
    Its consecutive return substitutions on letter $a$ are
    \begin{eqnarray*}
        \partial_a \sigma: & A \mapsto ABC,\ B \to AC,\ C \mapsto B \\
        \partial_a^2 \sigma: & 0 \mapsto 01,\ 1 \to 23,\ 2 \mapsto 013,\ 3 \mapsto 2
    \end{eqnarray*}
    Then, we have $\partial_a^3 \sigma = \partial_a^2 \sigma$.
    We conclude that $(X_\sigma, S)$ is aperiodic.
\end{ex}

\begin{ex}
    Consider the substitution $\sigma:
a \mapsto ab\ 
b \mapsto c\ 
c \mapsto abc
$.
Its derivative is $\partial_a \sigma : A \mapsto AA$ thus $(X_\sigma, S)$ is periodic.
\end{ex}

\begin{rem} \label{rem:per}
    If $\sigma$ is primitive periodic, then its Perron eigenvalue is an integer.
    Indeed, if $X_\sigma$ is of the form $\bigcup_{k = 0}^{\abs{w}-1} S^k w^\omega$, then $\ab(w)$ is a Perron eigenvector of $M_\sigma$
    since it is the limit of $\abs{w} \ab(\sigma^{n}(w_0))/\abs{\sigma^{n}(w_0)}$.
\end{rem}

\subsection{Semi-deciding recognizability} \label{ss:rec}

Given a substitution $\sigma$ and a morphism $\tau$, the fact that $\tau$ is recognizable in $X_\sigma$
is semi-decidable thank to the following.

\begin{prop}
    Let $\sigma$ be a primitive substitution and $\tau$ be a word morphism. Let $X_\sigma$ be the bi-infinite shift associated with $\sigma$.
    The substitution $\tau$ is recognizable in $X_\sigma$ if and only if
    \[
        \begin{split}
            \exists \epsilon > 0,\ \forall x,y \in X_\sigma,\ \forall 0 \leq p < \abs{\tau(x_0)},\ \forall 0 \leq q < \abs{\tau(y_0)}, \\
            d(S^{p} \tau(x), S^{q} \tau(y)) \leq \epsilon\ \Longrightarrow\ p=q \text{ and } x_0 = y_0.
        \end{split}
    \]
\end{prop}

In other words, we can uniquely desubstitute if we know the bi-infinite word in a large enough window.
It permits to semi-decide the recognizability of $\tau$ in $X_\sigma$ since
for each $\epsilon > 0$ there are only a finite number of possibilities to test.

\begin{proof}
    Let $\epsilon > 0$ be as in the last statement.
    Let us prove that $\tau$ is recognizable in $X_\sigma$.
    Let $x$, $y \in X_\sigma$, $0 \leq p < \abs{\tau(x_0)}$ and $0 \leq q < \abs{\tau(y_0)}$ be such that
    $S^p\tau(x) = S^q \tau(y)$.
    We have $d(S^p \tau(x), S^q \tau(y)) = 0 < \epsilon$
    so $p = q$ and $x_0 = y_0$.
    Then, if we shift $x$ and $y$,
    we have $d(\tau(S x), \tau(S y)) = 0$, thus $x_{1} = y_1$.
    And we can also shift in the other direction:
    we have $d(S^{-1} \tau(x), S^{-1} \tau(y)) = 0$, thus $x_{-1} = y_{-1}$.
    By induction, we get that $x = y$. Thus $\tau$ is recognizable in $X_\sigma$.
    
    Reciprocally, assume that such $\epsilon > 0$ does not exists.
    For each $\epsilon = \frac{1}{n} > 0$ it gives $x^{(n)}$ and $y^{(n)} \in X_\sigma$, $0 \leq p_n< \abs{\tau(x^{(n)}_0)}$
    and $0 \leq q_n < \abs{\tau(y^{(n)}_0)}$ such that
    $d(S^{p_n} \tau(x^{(n)}), S^{q_n} \tau(y^{(n)})) < 1/n$ and ($p_n \neq q_n$ or $x^{(n)}_0 \neq y^{(n)}_0$).
    Then, by compacity there exists subsequences $x^{(\phi(n))}$ and $y^{(\phi(n))}$ that converges and
    such that $x_0^{(\phi(n))}$, $y_0^{(\phi(n))}$, $p_{\phi(n)}$ and $q_{\phi(n)}$ are constant to $a$, $b$, $p$ and $q$ respectively, with either $a \neq b$ or $p \neq q$. 
    Let $x = \lim_{n \to \infty} x^{(\phi(n))}$ and $y = \lim_{n \to \infty} y^{(\phi(n))}$.
    We have $x,y \in X_\sigma$ since $X_\sigma$ is compact. Furthermore $S^p \tau(x) = S^q \tau(y)$ and ($p \neq q$ or $x_0 \neq y_0$) thus $\tau$ is not recognizable in $X_\sigma$.
\end{proof}

\section{A criterion for weak mixing} \label{sec:wm}

We use our description of the set of eigenvalues to deduce a characterization of the dimension of the set of eigenvalues.
Moreover, we provide a necessary and sufficient condition ensuring weak mixing,
for primitive unimodular substitutions without coboundary.
First, we need to introduce some notations.



Let $\B$ be a set of algebraic integers.
We denote $\gal(\B)$ the Galois group of $\K_\B : \Q$,
where $\K_\B$ is the splitting field of the product of minimal polynomials of elements of $\B$.

\begin{define}
    Let $\B$ be a set of algebraic integers.
    We define the \defi{graph} of $\B$ by
    \begin{itemize}
        \item vertices that are elements of $\B$ and their conjugates,
        \item edges between $\gamma$ and $\delta$ if and only if there exists $f \in \gal(\B)$
            such that $f(\gamma) \in \B$ and $f(\delta) \in \B$.
    \end{itemize}
    We say that $\B$ is \defi{connected} if its graph is connected.
\end{define}

Note that if $\beta$ has degree $d$ and $\gal(\{\beta\})$ contains a $d$-cycle $f$, then $\{\beta, f(\beta)\}$ is connected.
For example, if the degree of $\beta$ is a prime number and $\gamma$ is a conjugate of $\beta$, then $\{\beta, \gamma\}$ is connected.

Note that if $\B$ contains more than half of the conjugates of an algebraic integer, then
it is connected. For example, if $\beta$ is a Salem number, its set of conjugates of modulus $\geq 1$ is connected.

Note that if $\B$ is connected, then any set of algebraic integers containing $\B$ is connected.

\begin{thm}[Criterion for weak mixing] \label{crit:wm}
    Let $\sigma$ be a primitive aperiodic unimodular substitution with no non-trivial coboundary.
    Let $\tau$ be a word morphism recognizable in $X_\sigma$, and such that $\gcd((1...1)M_\tau) = 1$.
    Let $\B$ be the set of eigenvalues of $M_\sigma$ of modulus greater than or equal to $1$
    associated with at least one generalized eigenvector $v$ with $(1...1)M_\tau v \neq 0$.
    Then, the subshift $(X_{\tau \sigma^\omega},S)$ is weakly mixing if and only if $\B$ is connected.
\end{thm}

This criterion is a consequence of the following theorem that characterizes the dimension of the $\Q$-vector space spanned by eigenvalues of the subshift.

\begin{thm} \label{thm:dim:eigs}
    Let $\sigma$ be a primitive aperiodic substitution, and let $\tau$ be a word morphism recognizable in $X_\sigma$.
    Let $\pi_0 \pi_1 ... \pi_n$ be the decomposition of the minimal polynomial of $M_\sigma$ with
    $\pi_0$ the product of all cyclotomic factors, and with $\pi_i$, $1 \leq i \leq n$,
    the remaining powers of irreducible factors in $\Q$.
    Let
    \[
        \B = \{\beta \in \overline{\Q} \mid \exists i \in \{0,...,n\},\ (1...1) M_\tau G_i \neq \{0\},\ 
        \pi_i(\beta) = 0 \text{ and } \abs{\beta} \geq 1\},
    \]
    where $G_0 = R(X_\sigma) \cap \ker \pi_0(M_\sigma)$, 
    and $\forall i \geq 1$, $G_i = \ker \pi_i(M_\sigma)$.
    
    Then, the dimension of the $\Q$-vector space spanned by the eigenvalues of the subshift $(X_{\tau \sigma^\omega},S)$
    is the number of connected components of the graph of $\B$.
\end{thm}


\begin{proof} [Proof of Theorem~\ref{crit:wm}]
    Assume that $\B$ is connected.
    Since $\sigma$ is unimodular, the set of eigenvalues of the subshift is of the form $(W \cap \Z^\A) v_0$,
    where $v_0$ is the Perron eigenvector normalized such that $(1...1)M_\tau v_0 = 1$.
    By Theorem~\ref{thm:dim:eigs} and Lemma~\ref{lem:W:Wv}, the dimension of the $\Q$-vector space $W$ is $1$.
    Thanks to the hypothesis $\gcd((1...1)M_\tau) = 1$, we deduce that $W \cap \Z^\A = \Z (1...1)M_\tau$.
    Thus the set of eigenvalues of the subshift is $\Z$.
    
    Reciprocally, if the subshift is weakly mixing,
    then the $\Q$-vector space spanned by its eigenvalues has dimension $1$,
    thus by Theorem~\ref{thm:dim:eigs}, $\B$ is connected.
\end{proof}

\subsection{Proof of Theorem~\ref{thm:dim:eigs}}

This subsection is devoted to the proof of Theorem~\ref{thm:dim:eigs}.
We need the following proposition, which is a particular case of Proposition~16.7 in~\cite{Milne2015}.

\begin{prop} \label{prop:galois:inv}
    Let $F$ be a linear subspace of $\K^\A$ stable by the Galois group $G = \gal(\K:\Q)$.
    Then, $F = \K (F \cap \Q^\A)$. 
\end{prop}

    By Theorem~\ref{thm:car:eigs}, the $\Q$-vector space spanned by eigenvalues of $(X_{\tau\sigma^\omega}, S)$ is $W v$,
    where $v$ is the Perron eigenvector of sum $1$, and
    \[
        W = \Q^\A \cap (H \cap \K R(X_\sigma) \cap \Span_\K(V))^\circ,
    \]
    where $H = ((1...1)M_\tau)^\circ$ and $\K$ is the splitting field of the minimal polynomial of $M_\sigma$.
    
    \begin{lem} \label{lem:W:Wv}
        $\dim_\Q(W v) + \dim_\Q (E_0 + R^\circ) = \dim_\Q(W)$,
        where $E_0$ is the left generalized eigenspace of $M_\sigma$
        for the eigenvalue $0$ and $R = R(X_\sigma)$. 
    \end{lem}
    
    \begin{proof}
        Let $\varphi : W \to W v$ be the linear map defined by $\varphi(w) = wv$.
        This map is onto. Let us show that $\ker(\varphi) = E_0 + R^\circ$.
        Let $w \in \ker(\varphi)$.
        Then, by definition of $W$ we have $w(v-v') = 0$ so $wv' = 0$,
        for every generalized eigenvector $v' \in R$ associated with an eigenvalue of modulus $\geq 1$
        normalized such that $(1...1)M_\tau v' = 1$.
        Then, since $w \in \Q^\A$, we also have $wv'' = 0$ for every image $v''$ of $v'$ by a Galois element.
        Moreover, the Galois group act transitively on conjugates and stabilize $R$.
        Thus, we have $w \in E_0 + R^\circ$. The other inclusion is obvious.
    \end{proof}
    
    We decompose the non-contracting space as $\Span(V) = F_u \oplus F_1 \oplus ... \oplus F_n$,
    where $F_u$ is the sum of every generalized eigenspace for an eigenvalue root of unity,
    and $F_1$, ..., $F_n$ are the remaining generalized eigenspaces associated with eigenvalues of modulus $\geq 1$.
    By Lemma~\ref{lem:cob:unity}, $\K R$ contains $F_1 \oplus ... \oplus F_n$, thus
    \[
        \K R \cap \Span(V) = (\K R \cap F_u) \oplus F_1 \oplus ... \oplus F_n.
    \]
    
    \begin{lem}
        If $\K R \cap F_u \not\subset H$, then eigenvalues of $(X_{\tau \sigma^\omega},S)$ are rational.
    \end{lem}
    
    \begin{proof}
        Since the set $\K R \cap F_u$ is invariant by the Galois group $\gal(\K:\Q)$, there exists
        a rational vector $v \in R \cap F_u$ such that $(1...1) M_\tau v = 1$.
        Then, the set of eigenvalues of $X_{\tau \sigma^\omega}$ is included in $W v = \Q$.
    \end{proof}
    
    If $\K R \cap F_u \not\subset H$, then $\B$ contains a root of unity $\beta$.
    But it also contains every conjugate of $\beta$, so $\B$ is connected, and we get the wanted equality.
    
    In the following, we assume that $\K R \cap F_u \subset H$.
    Let $\B'$ be the set of conjugates of elements of $\B$, including $\B$
    (i.e. it is the set of vertices of the graph of $\B$).
    We choose a set $(v_\beta)_{\beta \in \B'}$ of vectors such that for every $\beta \in \B'$, $v_\beta$ is a generalized eigenvector associated with $\beta$, with $(1...1)M_\tau v_\beta = 1$,
    and such that $\forall \beta, \gamma \in \B'$, $\varphi(v_\beta) = v_\gamma$
    where $\varphi : \Q(\beta) \to \Q(\gamma)$ is the morphism of fields sending $\beta$ to $\gamma$.
    
    Let $D = \Span_\K \{v_\beta - v_\gamma \mid \beta, \gamma \in \B\}$. We have
    \[
        H \cap \K R(X_\sigma) \cap \Span_\K(V) = (\K R(X_\sigma) \cap F_u)
                                \oplus (F_1 \cap H) \oplus ... \oplus (F_n \cap H)
                                \oplus D
    \]
    thanks to the following.
    
    \begin{lem}
        Let $\varphi : E \to \K$ be a linear form and let $(A_i)_{i \in I}$ be a finite family of subspaces of $E$.
        Then,
        \[
            \ker(\varphi) \cap \bigoplus_{i \in I} A_i = \left( \bigoplus_{i \in I} \ker(\varphi) \cap A_i \right)
                    \oplus \Span_\K \{v_i - v_j \mid i,j \in I\},
        \]
        where for all $i \in I$, $v_i \in A_i$ is such that $\varphi(v_i) = 1$ if $\varphi(A_i) \neq \{0\}$, and $v_i = 0$ otherwise.
    \end{lem}
    
    \begin{proof}
        For each $i \in I$, let $B_i$ be a basis of $\ker(\varphi) \cap A_i$.
        Let $J = \{i \in I \mid \varphi(A_i) \neq \{0\}\}$.
        If $J = \emptyset$, then the result is obvious.
        Let $k \in J$.
        Since $\bigcup_{i \in I} B_i \cup \bigcup_{i \in J} \{v_i\}$ is a basis of $\bigoplus_{i \in I} A_i$,
        the family $\{v_k\} \cup \bigcup_{i \in I} B_i \cup \bigcup_{i \in J \backslash \{k\}} \{v_i - v_k\}$
        is also a basis of $\bigoplus_{i \in I} A_i$.
        Thus, $\bigcup_{i \in I} B_i \cup \bigcup_{i \in J \backslash \{k\}} \{v_i - v_k\}$ is a basis of $\ker(\varphi) \cap \bigoplus_{i \in I} A_i$ and $\bigcup_{i \in J \backslash \{k\}} \{v_i - v_k\}$ is a basis of $\Span_\K \{v_i - v_j \mid i,j \in I\}$.
    \end{proof}
    
    Let $G = \gal(\K:\Q)$.
    
    \begin{lem}
        Let $E$ be a subspace of $\K^n$.
        Then,
        \[
            \dim_\Q (\Q^n \cap E^\circ) = n - \dim_\K (G E) = n - \dim_\Q(\Q^n \cap G E).
        \]
    \end{lem}
    
    \begin{proof}
        It is a consequence of Proposition~\ref{prop:galois:inv}.
    \end{proof}
    
    Hence, to find the dimension of $W$, we have to find the dimension of $G (H \cap R(X_\sigma) \cap \Span_\K(V))$.
    For every $i \in \{1,...,n\}$ and every $g \in G$, $g(E_i \cap H) = g (E_i) \cap H$ since $H$ is stable by the Galois group $G$, thus
    \[
        G(H \cap \K R(X_\sigma) \cap \Span_\K(V)) = (\K R(X_\sigma) \cap F_u) \oplus \left( \sum_{g \in G} \sum_{i=1}^n H \cap g(F_i) \right) \oplus G D.
    \]
    Note that each $\K G_i$ is equal to some $\sum_{g \in G} g(F_j)$, and that $(1...1)M_\tau G_i = \{0\} \Longleftrightarrow (1...1)M_\tau F_j = \{0\}$.
    We can decompose the sum as a direct sum
    \[
        \sum_{g \in G} \sum_{i=1}^n H \cap g F_i = \left( \bigoplus_{\beta \in \B'} H \cap F_{\beta} \right)
            \oplus \bigoplus_{\beta \in \CCC} F_{\beta}
    \]
    whose dimension is $\abs{\A} - \abs{\B'} - \dim(F_u) - \dim(F_0)$, where $\CCC$ is the set of eigenvalues of $M_\sigma$
    not root of unity, not zero and not in $\B'$ and $F_\beta$ is the generalized eigenspace associated with $\beta$.
    Indeed, for every $\beta \in \B'$, $\dim(H \cap F_\beta) = \dim(F_\beta)-1$.
    
    \begin{lem}
        $\dim(G D) = \abs{\B'} - k$ where $k$ is the number of connected components of the graph of $\B$.
    \end{lem}
    
    \begin{proof}
        Let $c$ be a connected component of the graph of $\B$.
        Then, the dimension of $D_c = \Span \{v_\beta - v_\gamma \mid \beta \text{---} \gamma \text{ edge of } c \}$ is $\abs{c}-1$.
        Indeed, $D_c$ contains every difference $v_\beta - v_\gamma$ for $\beta, \gamma \in c$,
        and if we choose any $\beta \in c$ then $(v_\beta - v_\gamma)_{\gamma \in c \backslash \{\beta\}}$ is a basis of $D_c$.
        Moreover $GD = \Span\{v_\beta - v_\gamma \mid \beta \text{---} \gamma \text{ edge}\} = \bigoplus_{c} D_c$.
    \end{proof}
    
    We obtain
    \begin{eqnarray*}
        \dim W &=& \abs{\A} - \dim (\K R \cap F_u) - (\abs{\A} - \abs{\B'} - \dim(F_u) - \dim(F_0)) - (\abs{\B'} - k) \\
                &=& k + \dim(F_u) + \dim(F_0) - \dim (\K R \cap F_u) \\
                &=& k + \dim(E_0 + R^\circ),
    \end{eqnarray*}
    since
    $
        E_0 + R^\circ = \left( \K R \cap (F_u \oplus \bigoplus_{\beta \in \CCC} F_\beta ) \right)^\circ
                    = \left( (\K R \cap F_u) \oplus \bigoplus_{\beta \in \CCC} F_\beta \right)^\circ.
    $
    Thus, by Lemma~\ref{lem:W:Wv}, the dimension of the $\Q$-vector space
    spanned by eigenvalues of $(X_{\tau \sigma^\omega}, S)$
    is $\dim(W v) = k$. It finishes the proof of Theorem~\ref{thm:dim:eigs}.
    

\section{Examples} \label{sec:exs}

In this section, we give examples with non-trivial coboundaries and with explicit computation of eigenvalues.

\subsection{Constructing examples with non-trivial coboundary}

We provide a recipe for constructing examples with non-trivial coboundaries:

Begin by selecting any directed graph where edges are labeled with letters, ensuring that each letter appears only once.
Next, choose any map $f : V \to V$, where $V$ is the set of vertices.
Construct a substitution $\sigma$ as follow:
For each letter $a \in \A$,
we choose $\sigma(a)$ to be labels of a path from $f(s)$ to $f(e)$,
where $s \xrightarrow{a} e$ is the edge labeled by $a$ in the graph.

For example, consider the graph
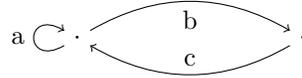
\begin{center}
    \begin{tikzpicture}
        \node at (0,0) (0) {$\cdot$};
        \node at (3cm,0) (1) {$\cdot$};
        
        \draw[->] (0) to [in=150, out=210, looseness=8] node[left]{a} (0);
        \draw[->] (0) to [bend left] node[below]{b} (1);
        \draw[->] (1) to [bend left] node[above]{c} (0);
    \end{tikzpicture}
\end{center}

We can construct the following substitutions preserving this graph:
\[
    \left\{\begin{array}{l}
    a \mapsto bc\\
    b \mapsto ab\\
    c \mapsto caa
    \end{array}, \begin{array}{l}
    a \mapsto caab\\
    b \mapsto c\\
    c \mapsto ab
    \end{array}, \begin{array}{l}
    a \mapsto bc\\
    b \mapsto aa\\
    c \mapsto bca
    \end{array}\right\}.
\]
Then, the coboundary $\begin{pmatrix} 0 & 1 & -1 \end{pmatrix}$ associated with the graph is an eigenvector of incidence matrices.
In the first case, it is associated with the eigenvalue $1$ because it preserves each vertex.
In the second case, the eigenvalue is $-1$ because it exchanges vertices.
In the third case, the eigenvalue is $0$ since both vertices are mapped to the same vertex.
Any directive sequence of substitutions constructed using this recipe for a given graph yields a subshift
that have coboundaries determined by the graph. 

\subsection{Examples of computation of eigenvalues} 

We give examples of computation of eigenvalues of morphic subshifts using the algorithm of Subsection~\ref{ss:algo}.

\begin{ex}[Thue-Morse]
    Let $\sigma: a \mapsto ab,\ b \mapsto ba$.
    Its incidence matrix $\begin{pmatrix} 1 & 1 \\ 1 & 1 \end{pmatrix}$ has only one eigenvalue of modulus $\geq 1$ which is $2$ associated with the eigenvector $v = \begin{pmatrix} 1/2 & 1/2 \end{pmatrix}^t$.
    There is no non-trivial coboundary, thus following the algorithm, $C V$ is a $1 \times 1$ invertible matrix, so
    $W = \Q^2$. Moreover, $\Delta_{M_\sigma} v = \Z\left[ \frac{1}{2} \right]$,
    thus the set of eigenvalues is $\Z \left[ \frac{1}{2} \right]$. The set $\B = \{2\}$ is connected.
\end{ex}

\begin{ex}[Regular paper folding]
    Let $\sigma: a \mapsto ab,\ b \mapsto cb,\ c \mapsto ad,\ d \mapsto cd$ and 
    $\tau : a \mapsto 11,\ b \mapsto 01,\ c \mapsto 10,\ d \mapsto 00$.
    Thanks to the algorithm of Subsection~\ref{ss:rec}, we can check that $\tau$ is recognizable in $X_\sigma$
    and we check also that $\sigma$ is not periodic.
    The set of coboundaries $R^\circ = \Span\{ \begin{pmatrix} 1 & -1 & 1 & -1 \end{pmatrix} \}$ of $X_\sigma$
    is included in the generalized eigenspace $E_0$ for the eigenvalue $0$,
    thus we take $C = \begin{pmatrix} 1 & 1 & 1 & 1 \end{pmatrix} M_\tau = \begin{pmatrix} 2 & 2 & 2 & 2 \end{pmatrix}$.
    We have $V = \begin{pmatrix} 1 & 0 \\ 1 & 1 \\ 1 & 0 \\ 1 & -1 \end{pmatrix}$ so $CV = \begin{pmatrix} 8 & 0 \end{pmatrix}$. Hence, $W$ is the left-kernel of $\begin{pmatrix} 0 & 1 & 0 & -1 \end{pmatrix}^t$.
    Since $C$ is a left-eigenvector of $M_\sigma$ for the eigenvalue $2$, the set of eigenvalues is
    $\Z\left[\frac{1}{2}\right] (W \cap \Z^4) v$ where $v = \frac{1}{4} \begin{pmatrix} 1 & 1 & 1 & 1 \end{pmatrix}^t$.
    It gives the set $\Z\left[\frac{1}{2}\right]$. $\B = \{1,2\}$ is connected.
\end{ex}

\begin{ex}[Constant-length with coboundary]
    Let $\sigma:
    a \mapsto aca\ 
    b \mapsto acb\ 
    c \mapsto cbc
    $.
    Its coboundary space is $R^\circ = \Span\{ \begin{pmatrix} 1 & 1 & -1 \end{pmatrix} \}$.
    The incidence matrix
    \[
        M = \left(\begin{array}{rrr}
        2 & 1 & 0 \\
        0 & 1 & 1 \\
        1 & 1 & 2
        \end{array}\right)
    \]
    has eigenvalues $2$ and $1$ and is not diagonalizable.
    Then, we have
    \[
        C = \begin{pmatrix} 1 & 1 & 1 \\ 1 & 1 & -1 \end{pmatrix} \quad \text{, } \quad V = \left(\begin{array}{rrr}
                                                                            1 & 0 & 1 \\
                                                                            0 & 1 & 1 \\
                                                                            -1 & -1 & 2
                                                                            \end{array}\right)
        \quad \text{ and } \quad CV = \begin{pmatrix} 0 & 0 & 4 \\ 2 & 2 & 0 \end{pmatrix}.
    \]
    We found that $W$ is the left-kernel of $\begin{pmatrix} 1 & -1 & 0 \end{pmatrix}^t$, that is the $\Q$-vector space spanned by $\{ \begin{pmatrix} 1 & 1 & 0 \end{pmatrix}, \begin{pmatrix} 0 & 0 & 1 \end{pmatrix} \}$.
    Let $v = \begin{pmatrix} 1/4 & 1/4 & 1/2 \end{pmatrix}$ be the Perron eigenvector of sum $1$.
    Since $\sigma$ has constant-length $3$, 
    the set of eigenvalues of $(X_\sigma, S)$ is $\Z\left[ \frac{1}{3} \right] (W \cap \Z^3)v = \frac{1}{2} \Z\left[ \frac{1}{3} \right]$.
\end{ex}

\begin{ex}[Weakly mixing subshift with coboundary]
    Let
    $\sigma: 
            a \mapsto d\ 
            b \mapsto ca\ 
            c \mapsto bd\ 
            d \mapsto abc
    $.
    Its matrix is
    \[
        \left(\begin{array}{rrrr}
        0 & 1 & 0 & 1 \\
        0 & 0 & 1 & 1 \\
        0 & 1 & 0 & 1 \\
        1 & 0 & 1 & 0
        \end{array}\right).
    \]
    The coboundary space is $R^\circ = \Span\{ \begin{pmatrix} 0 & 1 & -1 & 0 \end{pmatrix} \}$.
    Following the algorithm let $C = \begin{pmatrix} 1 & 1 & 1 & 1 \\ 0 & 1 & -1 & 0 \end{pmatrix}$.
    Then, the set of eigenvalues of modulus $\geq 1$ \linebreak is $\{-1, 2\}$.
    The generalized eigenspace for the eigenvalue $-1$ has dimension $2$, the matrix is not diagonalizable.
    The generalized vector $v = \begin{pmatrix} 1 & 1 & 1 & -2 \end{pmatrix}^t$ for the eigenvalue $-1$ is
    annihilated by the coboundary and is such that $\begin{pmatrix} 1 & 1 & 1 & 1 \end{pmatrix} v = 1$.
    Since $\Delta_M v = \Z$, the set of eigenvalues if $\Z$.
    We conclude that $(X_\sigma, S)$ is weakly mixing.
\end{ex}

\begin{ex}[Unimodular substitution with coboundary]
    Let $\sigma:
    a \mapsto aba\ 
    b \mapsto cb\ 
    c \mapsto cba
    $.
    Its coboundary space is $R^\circ = \Span\{ \begin{pmatrix} 1 & -1 & 0 \end{pmatrix}\}$.
    The characteristic polynomial of the incidence matrix is $(x-1)(x^2-3x+1)$.
    We have
    \[
        C = \begin{pmatrix}
                1 & 1 & 1 \\
                1 & -1 & 0
            \end{pmatrix}
        \quad , \quad
        V = \begin{pmatrix}
        1 & 1 \\
        1 & 0 \\
        1 - \varphi & -1
        \end{pmatrix}
        \quad \text{ and } \quad
        CV = \begin{pmatrix}
            1+\varphi & 0 \\ 0 & 1
        \end{pmatrix}
    \]
    The matrix $CV$ is invertible thus $W = \Q^3$.
    The matrix $M_\sigma$ is unimodular thus $\Delta_M = \Z^3$.
    The normalized Perron eigenvector is
    $v = \begin{pmatrix} 2-\varphi & 2 - \varphi & 2 \varphi - 3 \end{pmatrix}^t$
    where $\varphi$ is the golden ratio.
    Thus the set of eigenvalues is $\Z \left[ \varphi \right]$
\end{ex}

\begin{ex}[Reducible substitution with non-trivial graph of $\B$]
    Let  $\sigma: 1 \mapsto 16, 2 \mapsto 122, 3 \mapsto 12, 4 \mapsto 3, 5 \mapsto 124, 6 \mapsto 15$.
    The characteristic polynomial of the incidence matrix is $(x^3 - 3x^2 + 2x - 1) (x^3 - x - 1)$.
    The substitution has no non-trivial coboundary and $\B = \{\rho, \rho+1\}$
    where $\rho$ is the plastic ratio such that $\rho^3 = \rho + 1$.
    The graph of $\B$ has $3$ connected components of size $2$.
    The set of eigenvalues is $\Z \{1, 3\rho, 3\rho^2\}$.
\end{ex}

\begin{ex}[Irreducible substitution with non-trivial graph of $\B$]
    This example comes from~\cite{FMN96} (see also~\cite{Me23}).
    The substitution $\sigma : a \mapsto abdd, b \mapsto bc, c \mapsto d, d \mapsto a$ is irreducible,
    unimodular and has no non-trivial coboundary.
    The set $\B$ is $\{\beta, 1-\beta\}$ where $\beta$ is the greatest root of $x^4-2x^3-x^2+2x-1$.
    The Galois group $\gal(\B)$ has order $8$, and the graph of $\B$ has two connected components.
    The set of eigenvalues is $\Z[\sqrt{2}]$.
\end{ex}

\subsection{Examples coming from geometry}

In this subsection, we give examples of subshift that naturally arise from interval exchange transformations, or interval translation maps, or other geometrical problems.

\begin{ex}[$\S$-adic with coboundary coming from a problem of Novikov] \label{ex:CET4}
    The following substitutions
    \[
    \left\{
    \begin{array}{l}
    0 \mapsto 0\\
    1 \mapsto 1\\
    2 \mapsto 361\\
    3 \mapsto 3\\
    4 \mapsto 4\\
    5 \mapsto 5\\
    6 \mapsto 527\\
    7 \mapsto 7
    \end{array}, \begin{array}{l}
    0 \mapsto 527\\
    1 \mapsto 520\\
    2 \mapsto 0\\
    3 \mapsto 1\\
    4 \mapsto 361\\
    5 \mapsto 461\\
    6 \mapsto 4\\
    7 \mapsto 5
    \end{array}, \begin{array}{l}
    0 \mapsto 361\\
    1 \mapsto 3\\
    2 \mapsto 4\\
    3 \mapsto 427\\
    4 \mapsto 527\\
    5 \mapsto 7\\
    6 \mapsto 0\\
    7 \mapsto 360
    \end{array}, \begin{array}{l}
    0 \mapsto 4\\
    1 \mapsto 427\\
    2 \mapsto 527\\
    3 \mapsto 520\\
    4 \mapsto 0\\
    5 \mapsto 360\\
    6 \mapsto 361\\
    7 \mapsto 461
    \end{array}, \begin{array}{l}
    0 \mapsto 0\\
    1 \mapsto 052\\
    2 \mapsto 2\\
    3 \mapsto 3\\
    4 \mapsto 4\\
    5 \mapsto 614\\
    6 \mapsto 6\\
    7 \mapsto 7
    \end{array}, \begin{array}{l}
    0 \mapsto 0\\
    1 \mapsto 1\\
    2 \mapsto 2\\
    3 \mapsto 274\\
    4 \mapsto 4\\
    5 \mapsto 5\\
    6 \mapsto 6\\
    7 \mapsto 036
    \end{array}\right\}
    \]
    preserves the graph
    \begin{center}
        \begin{tikzpicture}
            \node at (0,0) (0) {$\cdot$};
            \node at (3cm,0) (1) {$\cdot$};
            
            \draw[->] (0) to [bend left] node[below]{$0,4$} (1);
            \draw[->] (0) to [loop left] node[left]{$1,7$} (0);
            \draw[->] (1) to [bend left] node[below]{$2,6$} (0);
            \draw[->] (1) to [loop right] node[right]{$3,5$} (1);
        \end{tikzpicture}
    \end{center}
    thus $\left(1,\,0,\,-1,\,0,\,1,\,0,\,-1,\,0\right)$ is a coboundary of the subshift for any primitive directive sequence.
    Our algorithm allows us to compute eigenvalues of the subshift for every primitive pre-periodic directive sequence.
    For example, the subshift of the second substitution has eigenvalues $\mathbb{Z}\left\{ \frac{1}{2} \beta^{2} + \frac{1}{2},\ \frac{1}{2} \beta^{2} + \frac{1}{2} \beta,\ \beta^{2} \right\}$ where $\beta$ is root of $x^{3} - 2 x^{2} - 1$.
    These substitutions come from a renormalization of a dynamical system on a non-orientable surface.
    The coboundary comes from a symmetry of the surface.
\end{ex}

\begin{ex}[Weakly mixing Interval Translation Maps]
    In~\cite{bruin2023}, it is shown that every primitive pre-periodic directive sequence with substitutions
    $1 \mapsto 2, 2 \mapsto 31^k, 3 \mapsto 31^{k-1}$, $k \geq 1$,
    gives a weakly mixing subshift. But their proof relies on awful calculations.
    Thanks to our criterion, we can give a different proof without calculation.
    There are no non-trivial coboundaries since any product of two of these substitutions is left-proper.
    They prove that the incidence matrix of the periodic part has two eigenvalues of modulus $\geq 1$
    (see Proposition~3.1 in~\cite{bruin2023}).
    Moreover, it is unimodular, so it is irreducible, and $\abs{\B} = 2 > 3/2$, so $\B$ is connected.
    By Theorem~\ref{crit:wm}, the subshift is weakly mixing.
\end{ex}

\begin{ex}[Arnoux-Yoccoz IET]
    The Arnoux-Yoccoz interval exchange transformation is defined by permutation $(2 5 4 7 6 3 1)$ and lengths
    \[
        \left(-\frac{1}{2} \beta + \frac{1}{2},\ \beta - \frac{1}{2},\ \frac{1}{2} \beta,\ 
        \frac{1}{2} \beta^{2},\ \frac{1}{2} \beta^{2},\ 
        -\frac{1}{2} \beta^{2} - \frac{1}{2} \beta + \frac{1}{2},\ 
        -\frac{1}{2} \beta^{2} - \frac{1}{2} \beta + \frac{1}{2}\right),
    \]
    where $\beta^3+\beta^2+\beta = 1$.
    See~\cite{LPV} for more details.
    
    This IET is pre-periodic for the Rauzy-Veech induction with pre-period $110110100101$ and period $101011010101001110000$.
    It corresponds respectively to substitutions
    $\tau: 1 \mapsto 15, 2 \mapsto 2, 3 \mapsto 1734335, 4 \mapsto 163434, 5 \mapsto 1635, 6 \mapsto 1634335, 7 \mapsto 1734$
    and
    $\sigma: 1 \mapsto 15172, 2 \mapsto 172, 3 \mapsto 1734365172, 4 \mapsto 15643472, 5 \mapsto 1565172, 6 \mapsto 1564365172, 7 \mapsto 173472$.
    Thank to our algorithm, we found that the set of eigenvalues of the IET is $\Z[\beta]$.
    This was expected since this IET is measurably conjugate to a translation on the torus $\TT^2$.
    There is no non-trivial coboundary and the graph of $\B = \{\beta\}$ has $3$ vertices and no edge.
    
    Note that if we consider the shift of $\sigma$, without the pre-period, then $\B$ contains $1$ so it is connected,
    thus by Theorem~\ref{crit:wm}, the shift of $\sigma$ is weakly mixing.
    
\end{ex}

\begin{ex}[Family of weakly mixing IET]
    This example comes from~4.3.1 in~\cite{DS}.
    Consider the family of interval exchange transformation be defined by permutation $(7654321)$
    and loop $10101(0^{k-1})10011100001111100000(1^{k-1})0$ for the Rauzy-Veech induction, with $k \geq 2$.
    The associated substitution
    \[
        1 \mapsto 1617,
        2 \mapsto 16252617,
        3 \mapsto 16253(4^{k-1})352617,
        4 \mapsto 16253(4^{k})352617,
    \]
    \[
        5 \mapsto 1625352617,
        6 \mapsto (1626)^{k}17,
        7 \mapsto (1626)^{k-1}17
    \]
    is unimodular, and it has no non-trivial coboundary since it is left-proper.
    We easily check that
    $
        (\begin{array}{ccccccc} 1 & -1 & 1 & -1 & 1 & -1 & 1\end{array})^t
    $
    is an eigenvector of the incidence matrix for the eigenvalue $1$.
    Hence, $\B$ contains $1$ so it is connected.    
    By Theorem~\ref{crit:wm}, the associated subshift is weakly mixing.
    
    Note that every pre-periodic directive sequence with these substitutions also give a weakly mixing subshift for the same reason.
\end{ex}

More examples of computations of eigenvalues of subshifts can be found here:
\url{https://www.i2m.univ-amu.fr/perso/paul.mercat/SubshiftsEigenvalues.pdf}
It uses Sage math and the package eigenmorphic. See end of Subsection~\ref{ss:algo} for more details.


\section{Acknowledgement}

I thank Pascal Hubert for his careful reading of this article and for interesting discussions that helped me to found Theorem~\ref{crit:wm} and Theorem~\ref{thm:dim:eigs}. I also thank Chanxi Wu for his collaboration to the recognizability result (see Subsection~\ref{ss:rec}).

\bibliographystyle{alpha}
\bibliography{references}

\end{document}

%% file: commands.tex
\newtheorem{thm}{Theorem}[section]
\newtheorem*{thm*}{Theorem}
\newtheorem{cor}[thm]{Corollary}
\newtheorem{lem}[thm]{Lemma}
\newtheorem{define}[thm]{Definition}
\newtheorem*{ex*}{Example}
\newtheorem{ex}[thm]{Example}
\newtheorem{prop}[thm]{Proposition}

\newtheorem{rem}[thm]{Remark}

\usepackage{tikz}

\newcommand\vide[1]{}

\def\restriction#1#2{\mathchoice
              {\setbox1\hbox{${\displaystyle #1}_{\scriptstyle #2}$}
              \restrictionaux{#1}{#2}}
              {\setbox1\hbox{${\textstyle #1}_{\scriptstyle #2}$}
              \restrictionaux{#1}{#2}}
              {\setbox1\hbox{${\scriptstyle #1}_{\scriptscriptstyle #2}$}
              \restrictionaux{#1}{#2}}
              {\setbox1\hbox{${\scriptscriptstyle #1}_{\scriptscriptstyle #2}$}
              \restrictionaux{#1}{#2}}}
\def\restrictionaux#1#2{{#1\,\smash{\vrule height .8\ht1 depth .85\dp1}}_{\,#2}}

\def\A{\mathcal{A}}
\def\B{\mathcal{B}}
\def\L{\mathcal{L}}

\def\TT{\mathbb{T}}

\def\F{\mathcal{F}}
\def\N{\mathbb{N}}
\def\Z{\mathbb{Z}}
\def\Q{\mathbb{Q}}
\def\R{\mathbb{R}}
\def\K{\mathbb{K}}
\def\RR{\mathcal{R}}
\def\C{\mathbb{C}}
\def\CCC{\mathcal{C}}
\def\1{\mathbbm{1}}
\def\id{\operatorname{id}}
\def\S{\mathcal{S}}

\def\NN{\mathcal{N}}
\def\M{\mathcal{M}}
\def\E{\mathcal{E}}

\def\ab{\operatorname{ab}}
\def\Span{\operatorname{span}}
\def\gal{\operatorname{Gal}}
\newcommand\abs[1]{\left| #1 \right|}
\newcommand\defi[1]{{\bf #1}}
\newcommand\norm[1]{\left\lVert #1 \right\rVert}
\def\cod{\operatorname{cod}}

\usepackage{listings}

%% file: main.bbl
\def\cprime{$'$}
\begin{thebibliography}{CDHM03}

\bibitem[Ada04]{Adam04}
Boris Adamczewski.
\newblock Symbolic discrepancy and self-similar dynamics.
\newblock {\em Annales de l'Institut Fourier}, 54(7):2201--2234, 2004.

\bibitem[BBY22]{BBY22}
Valérie Berthé, Paulina~Cecchi Bernales, and Reem Yassawi.
\newblock Coboundaries and eigenvalues of finitary s-adic systems, 2022.

\bibitem[BDM05]{BDM05}
XAVIER BRESSAUD, FABIEN DURAND, and ALEJANDRO MAASS.
\newblock Necessary and sufficient conditions to be an eigenvalue for linearly
  recurrent dynamical cantor systems.
\newblock {\em Journal of the London Mathematical Society}, 72(3):799--816,
  2005.

\bibitem[BR23]{bruin2023}
Henk Bruin and Silvia Radinger.
\newblock Interval translation maps with weakly mixing attractors, 2023.

\bibitem[BSTY19]{BSTY19}
Valérie Berthé, Wolfgang Steiner, Jörg~M. Thuswaldner, and Reem Yassawi.
\newblock Recognizability for sequences of morphisms.
\newblock {\em Ergodic Theory and Dynamical Systems}, 39(11):2896--2931, 2019.

\bibitem[CDHM03]{CDHM2003}
Maria~Isabel Cortez, Fabien Durand, Bernard Host, and Alejandro Maass.
\newblock Continuous and measurable eigenfunctions of linearly recurrent
  dynamical cantor systems.
\newblock {\em Journal of the London Mathematical Society}, 67(3):790--804,
  2003.

\bibitem[CFM08]{CFM08}
Julien Cassaigne, S\'ebastien Ferenczi, and Ali Messaoudi.
\newblock Weak mixing and eigenvalues for {Arnoux-Rauzy} sequences.
\newblock {\em Annales de l'Institut Fourier}, 58(6):1983--2005, 2008.

\bibitem[CS01]{CS}
Vincent Canterini and Anne Siegel.
\newblock Automate des pr\'efixes-suffixes associ\'e \`a une substitution
  primitive.
\newblock {\em Journal de th\'eorie des nombres de Bordeaux}, 13(2):353--369,
  2001.

\bibitem[DP20]{DP20}
Fabien Durand and Samuel Petite.
\newblock Conjugacy of unimodular pisot substitutions subshifts to domain
  exchanges, 2020.

\bibitem[DP22]{DurandPerrin2022}
Fabien Durand and Dominique Perrin.
\newblock {\em Dimension Groups and Dynamical Systems: Substitutions, Bratteli
  Diagrams and Cantor Systems}.
\newblock Cambridge Studies in Advanced Mathematics. Cambridge University
  Press, 2022.

\bibitem[DS16]{DS}
Hieu~Trung Do and Thomas~A. Schmidt.
\newblock New infinite families of pseudo-anosov maps with vanishing
  sah-arnoux-fathi invariant.
\newblock {\em Journal of Modern Dynamics}, 10(2):541--561, 2016.

\bibitem[Dur98]{Durand98}
Fabien Durand.
\newblock A characterization of substitutive sequences using return words.
\newblock {\em Discrete Mathematics}, 179(1):89--101, 1998.

\bibitem[Dur00]{Durand2000}
Fabien Durand.
\newblock Linearly recurrent subshifts have a finite number of non-periodic
  subshift factors.
\newblock {\em Ergodic Theory and Dynamical Systems}, 20(4):1061–1078, 2000.

\bibitem[Dur03]{Durand2003}
Fabien Durand.
\newblock Corrigendum and addendum to ‘linearly recurrent subshifts have a
  finite number of non-periodic factors’.
\newblock {\em Ergodic Theory and Dynamical Systems}, 23(2):663–669, 2003.

\bibitem[FD19]{DFM2019}
Alejandro~Maass Fabien~Durand, Alexander~Frank.
\newblock Eigenvalues of minimal cantor systems.
\newblock {\em Eur. Math. Soc}, 21(3):727--775, 2019.

\bibitem[FMN96]{FMN96}
S\'ebastien Ferenczi, Christian Mauduit, and Arnaldo Nogueira.
\newblock Substitution dynamical systems : algebraic characterization of
  eigenvalues.
\newblock {\em Annales scientifiques de l'\'Ecole Normale Sup\'erieure}, Ser.
  4, 29(4):519--533, 1996.

\bibitem[Hos86]{Host86}
B.~Host.
\newblock Valeurs propres des systèmes dynamiques définis par des
  substitutions de longueur variable.
\newblock {\em Ergod. Th. \& Dynam. Syst.}, 1986.

\bibitem[JHLV07]{LPV}
G.~Poggiaspalla J.~H.~Lowenstein and F.~Vivaldi.
\newblock Interval exchange transformations over algebraic number fields: the
  cubic arnoux–yoccoz model.
\newblock {\em Dynamical Systems}, 22(1):73--106, 2007.

\bibitem[MA20]{AkiyamaMercat2020}
Paul Mercat and Shigeki Akiyama.
\newblock {\em Yet Another Characterization of the Pisot Substitution
  Conjecture}, pages 397--448.
\newblock Springer International Publishing, Cham, 2020.

\bibitem[Mer23]{Me23}
Paul Mercat.
\newblock Geometrical representation of subshifts for primitive substitutions.
\newblock {\em Ergodic Theory and Dynamical Systems}, pages 1--28, 2023.

\bibitem[Mil15]{Milne2015}
J.S. Milne.
\newblock Algebraic geometry, 2015.

\bibitem[Mos96]{Mosse1996}
Brigitte Moss\'e.
\newblock Reconnaissabilit\'e des substitutions et complexit\'e des suites
  automatiques.
\newblock {\em Bulletin de la Soci\'et\'e Math\'ematique de France},
  124(2):329--346, 1996.

\end{thebibliography}
